\documentclass[11pt, english]{article}

\usepackage[margin= 2 cm,bottom=18mm, top= 18mm]{geometry}

\usepackage[square,sort,comma,numbers]{natbib}
\usepackage{amsthm}
\usepackage{amsmath}
\usepackage{amssymb}
\usepackage{setspace}
\usepackage{mathtools}
\usepackage{graphicx}
\graphicspath{ {./images/} }
\usepackage[hidelinks]{hyperref}
\usepackage{cleveref}
\usepackage{hyperref}
\usepackage{enumitem}
\usepackage{framed}
\usepackage{subcaption}
\usepackage{xspace}
\usepackage{thmtools} 
\usepackage{thm-restate}

\usepackage{thmtools}
\usepackage{thm-restate}

\usepackage{floatrow}

\usepackage{tikz}
\usepackage{mathdots}
\usepackage{xcolor}
\usepackage{diagbox}
\usepackage{colortbl}
\usepackage[absolute,overlay]{textpos}

\usetikzlibrary{calc}
\usetikzlibrary{decorations.pathreplacing}
\usetikzlibrary{positioning,patterns}
\usetikzlibrary{arrows,shapes,positioning}
\usetikzlibrary{decorations.markings}

\tikzstyle{edge}=[very thick]
\definecolor{bostonuniversityred}{rgb}{0.8, 0.0, 0.0}
\definecolor{arsenic}{rgb}{0.23, 0.27, 0.29}
\tikzstyle{diredge}=[postaction={decorate,decoration={markings,
		mark=at position .56 with {\arrow[scale = 1.3, black]{stealth};}}}]
\tikzset{
    arrow/.style={decoration={markings, mark=at position 0.5 with
    {\fill(-0.18*#1,-0.06*#1) -- (0,0) -- (-0.18*#1,0.06*#1) -- cycle;}, black}, postaction={decorate}},
    arrow/.default=1
}
\tikzset{
    arow/.style={decoration={markings, mark=at position 1 with
    {\fill(-0.09*#1,-0.03*#1) -- (0,0) -- (-0.09*#1,0.03*#1) -- cycle;}}, postaction={decorate}},
    arow/.default=1
}
\tikzset{
    arrrow/.style={decoration={markings, mark=at position 0.9 with
    {\fill(-0.09*#1,-0.03*#1) -- (0,0) -- (-0.09*#1,0.03*#1) -- cycle;}}, postaction={decorate}},
    arow/.default=1
}

\newcommand{\fitellipsis}[2] 
{\draw [fill=white]let \p1=(#1), \p2=(#2), \n1={atan2(\y2-\y1,\x2-\x1)}, \n2={veclen(\y2-\y1,\x2-\x1)}
    in ($ (\p1)!0.5!(\p2) $) ellipse [ x radius=\n2/2+0cm, y radius=1.1cm, rotate=\n1];
}
\newcommand{\Fitellipsis}[2] 
{\draw [fill=white]let \p1=(#1), \p2=(#2), \n1={atan2(\y2-\y1,\x2-\x1)}, \n2={veclen(\y2-\y1,\x2-\x1)}
    in ($ (\p1)!0.5!(\p2) $) ellipse [ x radius=\n2/2+0cm, y radius=1.4cm, rotate=\n1];
}

\theoremstyle{plain}

\newtheorem*{thm*}{Theorem}
\newtheorem{thm}{Theorem}[section]
\Crefname{thm}{Theorem}{Theorems}

\newtheorem*{lem*}{Lemma}
\newtheorem{lem}[thm]{Lemma}
\Crefname{lem}{Lemma}{Lemmas}

\newtheorem*{claim*}{Claim}
\newtheorem{claim}[thm]{Claim}
\crefname{claim}{Claim}{Claims}
\Crefname{claim}{Claim}{Claims}

\newtheorem{prop}[thm]{Proposition}
\Crefname{prop}{Proposition}{Propositions}

\Crefname{remar}{Remark}{Remarks}

\newtheorem{cor}[thm]{Corollary}
\crefname{cor}{Corollary}{Corollaries}

\newtheorem*{conj*}{Conjecture}
\newtheorem{conj}[thm]{Conjecture}
\crefname{conj}{Conjecture}{Conjectures}

\Crefname{qn}{Question}{Questions}

\newtheorem*{obs*}{Observation}
\newtheorem{obs}[thm]{Observation}
\Crefname{obs}{Observation}{Observations}

\Crefname{ex}{Example}{Examples}

\theoremstyle{definition}

\Crefname{prob}{Problem}{Problems}

\newtheorem{defn}[thm]{Definition}
\Crefname{defn}{Definition}{Definitions}

\theoremstyle{remark}

\renewenvironment{proof}[1][]{\begin{trivlist}
\item[\hspace{\labelsep}{\bf\noindent Proof#1.\/}] }{\qed\end{trivlist}}

\newcommand{\remove}[1]{}

\newcommand{\eps}{\varepsilon}

\title{\vspace{-1 cm}
Hamilton cycles in pseudorandom graphs}
\date{}

\author{
Stefan Glock\thanks{
Fakultät für Informatik und Mathematik, Universität Passau, Germany.
\emph{Email}: \textbf{stefan.glock@uni-passau.de}.
}
\and
David Munh\'a Correia\thanks{
Department of Mathematics, ETH, Z\"urich, Switzerland. Research supported in part by SNSF grant 200021\_196965.
\newline
\emph{Emails}: \textbf{\{david.munhacanascorreia, benjamin.sudakov\}@math.ethz.ch}.
}
\and
Benny Sudakov\footnotemark[2]}

\begin{document} 
\maketitle
\begin{abstract}
Finding general conditions which ensure that a graph is Hamiltonian is a central topic in graph theory. An old and well known conjecture in the area states that any $d$-regular $n$-vertex graph $G$ whose second largest eigenvalue in absolute value $\lambda(G)$ is at most $d/C$, for some universal constant $C>0$, has a Hamilton cycle. In this paper, we obtain two main results which make substantial progress towards this problem. Firstly, we settle this conjecture in full when the degree $d$ is at least a small power of $n$. Secondly, in the general case we show that $\lambda(G) \leq  d/ C(\log n)^{1/3}$ implies the existence of a Hamilton cycle, improving the 20-year old bound of $d/ \log^{1-o(1)} n$ of Krivelevich and Sudakov. We use in a novel way a variety of methods, such as a robust P\'osa rotation-extension technique, the Friedman-Pippenger tree embedding with rollbacks and the absorbing method, combined with additional tools and ideas.

Our results have several interesting applications, giving best
bounds on the number of generators which guarantee the Hamiltonicity of random Cayley graphs, 
which is an important partial case of the well known Hamiltonicity conjecture of Lov\'asz. They can also be used to improve a result of Alon and Bourgain on additive patterns in multiplicative subgroups.
\end{abstract}

\section{Introduction}\label{sec:intro}

A \emph{Hamilton cycle} in a graph $G$ is a cycle passing through all the vertices of~$G$. If it exists, then $G$ is called \emph{Hamiltonian}. Being one of the most central notions in Graph Theory, it has been
extensively studied by numerous researchers, see e.g., \cite{ajtai1985first,bollobas1987algorithm,chvatal1972note,HKS:09,kuhn2013hamilton,kuhn2014proof,krivelevich2011critical,krivelevich2014robust,MR3545109,ferber2018counting, cuckler2009hamiltonian, posa:76}, and the surveys \cite{gould2014recent, MR3727617}. In particular, the problem of deciding Hamiltonicity of
a graph is known to be NP-complete and thus, finding general conditions which ensure that $G$ has a Hamilton cycle is one of the most popular topics in Graph Theory. For instance, two famous theorems of this nature are the celebrated result of Dirac \cite{dirac1952some}, which states that if the minimum degree of an $n$-vertex graph $G$ is at least $n/2$, then $G$ contains a Hamilton cycle, and the criterion of Chv\'atal and Erd\H{o}s \cite{chvatal1972note} that a graph is Hamiltonian if its connectivity number is at least as large as its independence number. 

In fact, most of the classical criteria for Hamiltonicity focus on rather dense graphs. A prime example of this is clearly Dirac's theorem stated above, but also the Chv\'atal-Erd\H{o}s condition requires the graph to be relatively dense, of average degree $\Omega(\sqrt{n})$. In contrast, sufficient conditions that ensure Hamiltonicity of sparse graphs seem much more difficult to obtain. A natural starting point towards this topic is to consider sparse random graphs, to which a lot of research has been dedicated in the last 50 years. In a breakthrough paper in 1976, P\'osa~\cite{posa:76} proved that the binomial random graph model $G(n,p)$ with $p \geq C \log n/n$ for some large constant $C$ almost surely contains a Hamilton cycle. In doing so, he invented the influential \emph{rotation-extension technique} for finding long cycles and paths, which has found numerous further applications since then. P\'osa's result was later refined by Korshunov \cite{korshunov} and in 1983, a more precise threshold for Hamiltonicity was obtained by Bollob\'as \cite{bollobas1984evolution} and Koml\'os and Szemer\'edi \cite{komlos1983limit}, who independently showed that if $p = (\log n + \log \log n + \omega(1))/n$, then $G(n,p)$ is almost surely Hamiltonian. It is a standard exercise to note that this is essentially tight - indeed, if $p = (\log n + \log \log n - \omega(1))/n$, then $G(n,p)$ almost surely has a vertex with degree at most $1$, and hence is not Hamiltonian. In parallel, significant attention has also been given to the Hamiltonicity of the random $d$-regular graph model $G_{n,d}$ - it is known that $G_{n,d}$ almost surely contains a Hamilton cycle for all values of $3 \leq d \leq n-1$. For this result, the reader is referred to Cooper, Frieze and Reed~\cite{cooper2002random} and Krivelevich, Sudakov, Vu and Wormald~\cite{krivelevich2001random} and their references.

Given the success of the study of Hamilton cycles in sparse random graphs, it became natural to then consider pseudorandom graphs, which are deterministic graphs that resemble random graphs in various important properties. A convenient way to express pseudorandomness is via spectral techniques and was introduced by Alon. An \emph{$(n,d,\lambda)$-graph} is an $n$-vertex $d$-regular graph $G$ whose second largest eigenvalue in absolute value, $\lambda(G)$, is such that $\lambda(G)\le \lambda$. Roughly speaking, $\lambda(G)$ is a measure of how ``smooth'' the edge-distribution of $G$ is, and the smaller its value, the closer to ``random'' $G$ behaves. The reader is referred to \cite{krivelevich2006pseudo} for a detailed survey concerning pseudorandom graphs.

In a rather influential paper, Krivelevich and Sudakov~\cite{KS:03} employed P\'osa's rotation-extension technique to prove the very general result that $(n,d,\lambda)$-graphs are Hamiltonian, provided $\lambda$ is significantly smaller than~$d$. Precisely, they showed that if $n$ is sufficiently large, then
\begin{align}
		d/\lambda  \ge \frac{1000\log n (\log\log\log n)}{(\log\log n)^2} \label{spectral ratio KS}
	\end{align}
guarantees that any $(n,d,\lambda)$-graph contains a Hamilton cycle.

This result has found numerous applications in the last 20 years towards some well-known problems, some of which we will discuss later. Given its significance and generality, it leads to the very natural and fundamental question of whether a smaller ratio of $d/\lambda$ is already sufficient to imply Hamiltonicity. Krivelevich and Sudakov~\cite{KS:03} conjectured that it should suffice that $d/\lambda$ is only a large enough constant. 
\begin{conj}\label{conj:KS}
	There exists an absolute constant $C >0$ such that any $(n,d,\lambda)$-graph with $d/\lambda \ge C$ contains a Hamilton cycle.
\end{conj}

\noindent Despite the plethora of incentives, there has been no improvement until now on the Krivelevich and Sudakov bound stated in \eqref{spectral ratio KS}. 
In this paper, we make significant progress towards Conjecture~\ref{conj:KS} in two ways. First, we improve on the Krivelevich and Sudakov bound in general by showing that a spectral ratio of order $(\log n)^{1/3}$ already guarantees Hamiltonicity.

\begin{thm}\label{thm:main general}
There exists a constant $C>0$ such that any $(n,d,\lambda)$-graph with $d/\lambda \ge C(\log n)^{1/3}$ contains a Hamilton cycle.
\end{thm}
\noindent The proof of the above result will rely on the P\'osa rotation-extension method with various new ideas. Namely, we will need to develop some techniques in order to use this method in a \emph{robust} manner. The reader is referred to Section~\ref{sec:outline} for an outline of the proof. 

Secondly, we confirm Conjecture \ref{conj:KS} in full when the degree is polynomial in the order of the graph.

\begin{thm}\label{thm:maindense}
For every constant $\alpha > 0$, there exists a constant $C > 0$ such that any $(n,d,\lambda)$-graph with $d \geq n^{\alpha}$ and $d/\lambda  \geq C$ contains a Hamilton cycle.  
\end{thm}
\noindent 
In fact, Theorem \ref{thm:maindense} is a corollary of a more general statement that we will prove (Theorem \ref{thm:pflhamiltonian}). This in particular states that $(n,d,\lambda)$-graphs with linearly many vertex-disjoint cycles are Hamiltonian.  

Before discussing the applications of both of our main results, we refer the reader to Section~\ref{sec:outline} for an outline of the proof of these, as well as, of the structure of the paper.

\subsection{Applications and related problems}\label{sec:appsintro}

Both Theorem \ref{thm:main general} and Theorem \ref{thm:maindense} immediately yield improvements in several applications which made use of the result of Krivelevich and Sudakov. These will be discussed in a detailed manner in Section~\ref{sec:apps}. One application is an important special case of a famous open question of 
Lov\'asz \cite{lovasz1969combinatorial}  from 1969 concerning the Hamiltonicity of a certain class of well-behaved graphs (see e.g., \cite{curran1996hamiltonian} and its references for more background on the problem).
\begin{conj}\label{conj:lovasz}
Every connected vertex-transitive graph contains a Hamilton path, and, except for five known examples, a Hamilton cycle.
\end{conj}
\noindent Since Cayley graphs are vertex-transitive and none of the five known exceptions in Lov\'asz's conjecture is a Cayley graph, the conjecture in particular includes the following, which was asked much earlier in 1959 by Rapaport Strasser~\cite{rapaportstrasser:59}.
\begin{conj}\label{conj:transitive}
Every connected Cayley graph is Hamiltonian.
\end{conj}
\noindent For these conjectures, a proof is currently out of sight. Indeed, notable progress towards them in their full generality are a result of Babai \cite{babai1979long} that every vertex-transitive $n$-vertex graph contains a cycle of length $\Omega(\sqrt{n})$ (see~\cite{deVos} for a recent improvement) and a result of Christofides, Hladk\'y and M\'ath\'e \cite{christofides2014hamilton} that every vertex-transitive graph of linear minimum degree contains a Hamilton cycle. 

Given this, it is natural to consider the ``random'' version of Conjecture~\ref{conj:transitive}. Indeed, Alon and Roichman~\cite{AR:94} showed that in any group $G$, a random set $S$ of $O(\log |G|)$ elements is such that the Cayley graph generated by them, $\Gamma(G,S)$, is almost surely connected. Therefore, a particular instance of Conjecture \ref{conj:transitive} is to show that $\Gamma(G,S)$ is almost surely Hamiltonian, which is itself a conjecture of Pak and Radoi\v{c}i\'{c}~\cite{PR:09}. In fact, this relates directly to Conjecture \ref{conj:KS} since it can be shown, generalizing the result of Alon and Roichman, that if $|S| \geq C \log |G|$ for some large constant $C$, then $\Gamma(G,S)$ is almost surely an $(n,d,\lambda)$-graph with $d/ \lambda \geq K$ for some large constant $K$. Hence, Conjecture~\ref{conj:KS} would imply the Hamiltonicity of $\Gamma(G,S)$. In Section~\ref{sec:random cayley}, improving on several earlier results \cite{MH:96, KS:03, CM:12}  we will show how Theorem~\ref{thm:main general} can be used to prove that if $|S|$ is of order $\log^{5/3} n$, then $\Gamma(G,S)$ is almost surely Hamiltonian (see Theorem~\ref{thm:cayley random}). In the same section, we will also give an improved bound on a related problem of Akbari, Etesami, Mahini, and
Mahmoody \cite{AB:14} concerning Hamilton cycles in coloured complete graphs which use only few colours.

Another application of our results concerns one of the central themes in Additive Combinatorics, the interplay between the two operations sum and product. A well-known fact in this area is that any multiplicative subgroup $A$ of the finite field $\mathbb{F}_q$ of size at least $q^{3/4}$ must contain two elements $x,y$ such that $x+y$ also belongs to $A$. Motivated by this, Alon and Bourgain~\cite{AB:14} studied more complex additive structures in multiplicative subgroups. In particular, they proved that when a subgroup has size $|A|\ge q^{3/4} (\log q)^{1/2-o(1)}$, then there is a cyclic ordering of the elements of $A$ such that the sum of any two consecutive elements is also in $A$. Using Theorem \ref{thm:maindense}, we can improve on Alon and Bourgain's result, showing that the additional polylog-factor can be avoided. This shows that when $|A|$ is of order $q^{3/4}$, not only does it contain $x,y, x+y \in A$ but also much more complex structures. 

Finally, in Section~\ref{sec:pak}, we give an application of our techniques to another problem related to Conjecture \ref{conj:transitive}. Motivated by this conjecture, Pak and Radoi\v{c}i\'{c} \cite{PR:09} showed that every group $G$ has a set of generators $S$ of size at most $\log_2 |G|$ such that the Cayley graph $\Gamma(G,S)$ is Hamiltonian, which is optimal since there are groups that do not have generating sets of size smaller than $\log_2 |G|$. Since their proof relies on the classification of finite simple groups, they asked to find a classification-free proof of this result. 
Using the methods developed for the proof of Theorem~\ref{thm:maindense} we give a classification-free proof that there is always such a set $S$ with $|S| = O(\log n)$.

\section{Outline}\label{sec:outline}
The proofs of Theorems \ref{thm:main general} and \ref{thm:maindense} are quite different and in this section we give an outline of the main ideas in these proofs. 
\subsection{Robust P\'osa rotation and Theorem \ref{thm:main general}}
\noindent As mentioned in the introduction, the proof of Theorem \ref{thm:main general} relies on P\'osa's rotation-extension technique. This is a method used for finding Hamilton cycles in connected graphs. The general argument takes a longest path $P$ in an expanding graph $G$ and uses the rotation-extension technique to \emph{close the path $P$ into a cycle}, that is, it finds a cycle $C$ on the same vertex set $V(P)$. If $|P| = |G|$, then $C$ is a Hamilton cycle, as desired. Otherwise, since $G$ is connected there is a vertex $v \in V(G) \setminus C$ and an edge from $v$ to $C$; this would give a longer path $P'$ than $P$, contradicting the maximality of $P$. Clearly the crucial step of the argument above is in \emph{closing the path $P$ into a cycle} which is where the rotation-extension technique comes into play. Let us then briefly describe this beautiful method introduced by P\'osa (a much more detailed take on it is made in Section \ref{sec:standardposarot}). 

Informally, given a path $P = (v_1,v_2, \ldots, v_l)$, an edge $v_iv_l$ allows one to construct a new path $P' = (v_1,v_2, \ldots, v_i, v_l, v_{l-1}, \ldots, v_{i+1})$ which also has $v_1$ as an endpoint. This is called a \emph{rotation of} $P$ with \emph{fixed endpoint} $v_1$ and note that it ``creates a new endpoint $v_{i+1}$'' in that $P'$ is a $v_1v_{i+1}$-path also with vertex set $V(P)$. We can then continue this procedure, rotating $P'$ also with fixed endpoint $v_1$ to create a new path $P''$ which has a new endpoint $v_j$, and so on. If a path $Q$ is the result of $t$ consecutive such rotations, we will refer to it as a $t$\emph{-rotation of $P$} with fixed endpoint $v_1$. Notice also that if $P$ is a longest path in $G$, then each one of these \emph{new endpoints} must have all their neighbours in $V(P)$ - otherwise, a neighbour outside of $V(P)$ would create a path longer than $P$. Using this fact we can then create many new paths and endpoints. Indeed, since our graph $G$ has nice expansion properties we can, from starting at a longest path $P$ and fixing one of its endpoints $v_1$, find $\Omega(n)$ vertices $v_j$ such that there is a $v_1v_j$-path also on the vertex set $V(P)$ which is a $t$-rotation of $P$ for some $t = O(\log n) $ (for a precise statement of this, see Lemma \ref{lem:usualrotation1} and Corollary \ref{cor:usualrotation}). 

Our goal is now to close $P$ into a cycle, that is, find a spanning cycle on the vertex set $V(P)$. A reasonable first approach is as follows. By the previous paragraph, the rotation technique allows to find a set $X$ of size $\Omega(n)$ such that every vertex $v_j \in X$ is such that there is a $v_1v_j$-path $P_j$ with vertex set $V(P)$. Similarly, for each $v_j \in X$, we can find a set $Y_j$ of size $\Omega(n)$ such that every vertex $v_k \in Y_j$ is such that there is a $v_kv_j$-path $P_{k,j}$ with vertex set $V(P)$. Clearly, the set $Y_j$ is dependent on the vertex $v_j \in X$ which is chosen. Nevertheless, it might be that there exists a large set $X' \subseteq X$ such that the intersection of all sets $Y_j$ with $v_j \in X'$ is also large. Then, denoting this intersection by $Y'$, it will be the case that for each $a \in X',b \in Y'$ there is an $ab$-path with vertex set $V(P)$. Using that $G$ is a pseudorandom graph, we can then find an edge $ab$ with $a \in X'$ and $b \in Y'$ - this creates a cycle as desired. For example, one can observe that such $X',Y'$ can be found if we can: rotate consecutively with fixed endpoint $v_1$ and only using endpoints in $\{v_{l/2+1}, \ldots, v_{l-1},v_l\}$, that is, only using edges induced by these vertices; and then rotate consecutively with fixed endpoint $v_l$ (or any endpoint in $\{v_{l/2+1}, \ldots, v_{l-1},v_l\}$) and only using endpoints in $\{v_{1}, \ldots, v_{l/2}\}$. Indeed, then since the two sets of rotations do not interfere with each other, we will get two large sets $X' \subseteq \{v_{l/2+1}, \ldots, v_{l-1},v_l\}, Y' \subseteq \{v_{1}, \ldots, v_{l/2}\}$ such that for each $a \in X',b \in Y'$ there is a desired $ab$-path. However, it is clearly not the case in general that such a situation will hold. This is the main issue that we have to resolve. Nevertheless, the objective is to, using rotations, find disjoint sets $X,Y$ which are large enough, such that for every $a \in X, b \in Y$, there is an $ab$-path with vertex set $V(P)$. Then, by pseudorandomness there is an edge between $X,Y$ which will create the desired cycle. 

As we saw with the example in the previous paragraph, in order to do this, it will be convenient to control the way we rotate. With this in mind, we will first \emph{clean} the path $P$ (this is done in Lemma \ref{lem:findingcleanpartition}). Informally, if $G$ is an $(n,d,\lambda)$-graph and $P$ is large, we can find a collection of disjoint intervals in the path which cover almost all the vertices of $P$ such that almost every vertex $v$ in each of these intervals has \emph{many} neighbours in \emph{several} other intervals. Clearly, this will allow us to rotate in various different manners, avoiding certain situations. Indeed, let us informally describe how the proof goes after this cleaning is done. Suppose for simplicity that $P$ is fully partitioned into $k$ intervals of size $|P|/k$ and that \emph{all} the vertices have many neighbours in at least $\eps k$ intervals, for some appropriate $\eps$. We now fix a large sub-path $A$ at the start of the path containing the vertex $v_1$ which contains roughly $\eps k/2$ intervals. Observe already that since $A$ has that size, every vertex has lots of neighbours outside of $A$ (it has many neighbours in at least $\eps k/2$ intervals which are disjoint to $A$). This implies that we can always rotate from any vertex while avoiding $A$ - let us denote this property by $\mathbf{P}$. Also, using the pseudorandomness of $G$ we can find a subset $A' \subseteq A$ of size $0.99|A|$ which has large minimum degree and has the property that every vertex in $A'$ has lots of neighbours such that both of their neighbours in $P$ also belong to $A'$. Note that this implies that if we are doing rotations and our current endpoint belongs to $A'$, then we can continue rotating while ``staying'' in $A' \subseteq A$, that is, only creating new endpoints which belong to $A'$ - let us call this property $\mathbf{Q}$. 

Now, using property $\mathbf{P}$, we can first rotate with fixed endpoint $v_l$ while avoiding $A$ to find many new endpoints. Then, the pseudorandomness of $G$ guarantees that we can find an edge between one of these endpoints and $A'$, giving a new endpoint $z \in A'$. Moreover, crucially note that since we only ``touched'' $A$ at this last step, we now have a $zv_l$-path $P'$ such that $A$ is broken into only two disjoint sub-paths of $P'$ (one at the start containing $z$ and one at the middle - see Figure \ref{fig:outline} where these paths are depicted in red) and thus, in particular, $A'$ still has property $\mathbf{Q}$. Let us denote the interval of $P'$ between these two sub-paths of $A$ as $I$. Using property $\mathbf{Q}$ we can then start rotating $P'$ with fixed endpoint $v_l$ while staying in $A' \subseteq A$. Let the set of endpoints created as such be denoted as $X$, so that for each $ x \in X$ we created an $xv_l$-path $P_x$. Since $P' \setminus A$ was avoided in these rotations, observe that the only difference between $P_x \setminus A$ and $P_{x'} \setminus A$ (for two different $x,x' \in X$) in the paths $P_x,P_{x'}$ is the direction in which the interval $I$ is traversed in the orientation $x,x' \rightarrow v_l$ (see Figure \ref{fig:outline} for an illustration). Then, by possibly reducing the size of $X$ by half, we can assume that this direction is the same for all endpoints in $X$ - let us denote this property by $\mathbf{R}$. Finally, we can now, for some $x \in X$, use property $\mathbf{P}$ to rotate with fixed endpoint $x$ only using $P' \setminus A$. Then, since these rotations avoid $A$, property $\mathbf{R}$ implies that every endpoint $y$ this creates is also an endpoint for all other $x' \in X$, that is, for all $x' \in X$ there is an $x'y$-path on the vertex set $V(P)$. Letting $Y$ denote this set of endpoints $y$, if $X,Y$ are large enough then the pseudorandomness of $G$ implies that there is an edge $xy$ with $x \in X, y \in Y$, thus creating a cycle as desired.
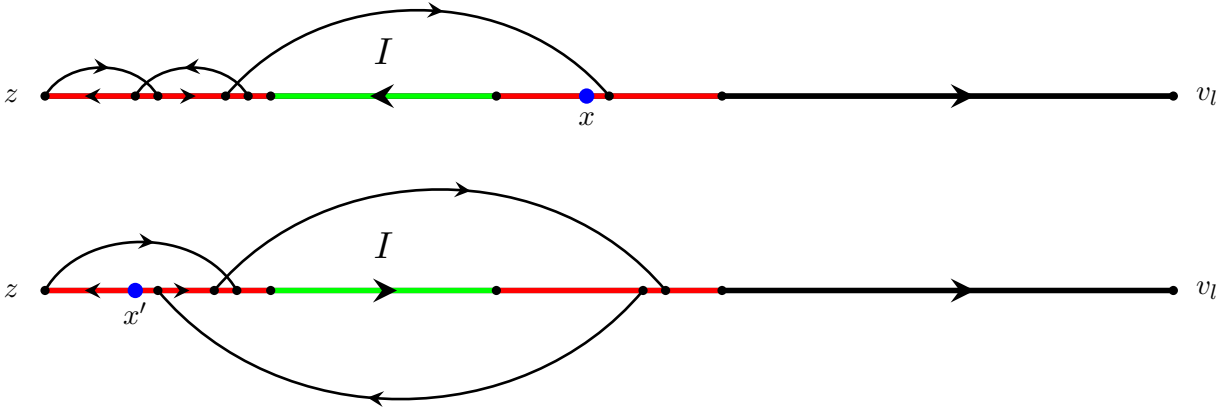
\begin{figure}[ht]
    \centering
    \begin{tikzpicture}[scale=1.5,main node/.style={circle,draw,color=black,fill=black,inner sep=0pt,minimum width=3pt}]
        \tikzset{cross/.style={cross out, draw=black, fill=none, minimum size=2*(#1-\pgflinewidth), inner sep=0pt, outer sep=0pt}, cross/.default={2pt}}
	\tikzset{rectangle/.append style={draw=brown, ultra thick, fill=red!30}}
	    
\draw[line width= 2 pt] (-10,0) to (0,0);	 

\draw[red, line width= 2 pt] (-10,0) to (-8,0);
\draw[diredge, green, line width= 2 pt] (-6,0) to (-8,0);
\draw[red, line width= 2 pt] (-6,0) to (-4,0);
\draw[diredge, line width= 1 pt] (-10,0) to [bend left=60](-9,0);

\draw[diredge, red, line width= 1 pt] (-9.2,0) to (-10,0);
\draw[diredge, line width= 1 pt] (-8.2,0) to [bend right=60](-9.2,0);

\draw[diredge, line width= 1 pt] (-8.4,0) to [bend left=50](-5,0);
\draw[diredge, line width= 2 pt] (-4,0) to (0,0);
\draw[diredge, red, line width= 1 pt] (-9,0) to (-8.4,0);
\node[scale = 1.3] at (-7,0.4) {$I$};



\node[scale = 3] at (-5,0) {$.$};
\node[blue, scale = 5] at (-5.2,0) {$.$};
\node[scale = 1] at (-5.2,-0.2) {$x$};

\node[scale = 3] at (-10,0) {$.$};
\node[scale = 1] at (-10.3,0) {$z$};
\node[scale = 3] at (0,0) {$.$};
\node[scale = 1] at (0.3,0) {$v_l$};
\node[scale = 3] at (-8.2,0) {$.$};
\node[scale = 3] at (-8,0) {$.$};
\node[scale = 3] at (-6,0) {$.$};
\node[scale = 3] at (-4,0) {$.$};
\node[scale = 3] at (-9.2,0) {$.$};
\node[scale = 3] at (-9,0) {$.$};
\node[scale = 3] at (-8.4,0) {$.$};

    \end{tikzpicture}

    \begin{tikzpicture}[scale=1.5,main node/.style={circle,draw,color=black,fill=black,inner sep=0pt,minimum width=3pt}]
        \tikzset{cross/.style={cross out, draw=black, fill=none, minimum size=2*(#1-\pgflinewidth), inner sep=0pt, outer sep=0pt}, cross/.default={2pt}}
	\tikzset{rectangle/.append style={draw=brown, ultra thick, fill=red!30}}

\draw[line width= 2 pt] (-10,0) to (0,0);	 

\draw[red, line width= 2 pt] (-10,0) to (-8,0);
\draw[diredge, green, line width= 2 pt] (-8,0) to (-6,0);
\draw[red, line width= 2 pt] (-6,0) to (-4,0);
\draw[diredge, line width= 1 pt] (-10,0) to [bend left=60](-8.3,0);

\draw[diredge, red, line width= 1 pt] (-9.2,0) to (-10,0);
\draw[diredge, line width= 1 pt] (-8.5,0) to [bend left=50](-4.5,0);

\draw[diredge, line width= 1 pt] (-4.7,0) to [bend left=50](-9,0);
\draw[diredge, line width= 2 pt] (-4,0) to (0,0);
\draw[diredge, red, line width= 1 pt] (-9,0) to (-8.5,0);
\node[scale = 1.3] at (-7,0.4) {$I$};

\node[scale = 3] at (-9,0) {$.$};
\node[blue, scale = 5] at (-9.2,0) {$.$};
\node[scale = 1] at (-9.2,-0.2) {$x'$};

\node[scale = 3] at (-10,0) {$.$};
\node[scale = 1] at (-10.3,0) {$z$};
\node[scale = 3] at (0,0) {$.$};
\node[scale = 1] at (0.3,0) {$v_l$};
\node[scale = 3] at (-8.3,0) {$.$};
\node[scale = 3] at (-8.5,0) {$.$};
\node[scale = 3] at (-8,0) {$.$};
\node[scale = 3] at (-6,0) {$.$};
\node[scale = 3] at (-4,0) {$.$};
\node[scale = 3] at (-4.5,0) {$.$};
\node[scale = 3] at (-4.7,0) {$.$};

    \end{tikzpicture}
    \caption{The $zv_l$-path $P'$ is depicted, the set $A$ is pictured in red and the interval $I$ in green. Two rotations with fixed endpoint $v_l$ which ``stay'' in $A$ (using property $\mathbf{Q}$) are depicted with the interval $I$ being traversed in different directions in the resulting paths $P_x,P_{x'}$ which are depicted using arrows in the direction $x,x' \rightarrow v_l$.}
    \label{fig:outline}
\end{figure}

\noindent Clearly, the previous discussion is only informal and lacks a lot of important details that we have to deal with in order to do the proof. Namely, at the start it will not be the case that we can partition all of $P$ into intervals so that every vertex has many neighbours in several intervals (and so, in particular, we will not have property $\mathbf{P}$). In particular, the endpoints of the path $P$ will not necessarily have neighbours in many intervals. Therefore, we will first have to perform some rotations in order to get endpoints which do satisfy this. This will however imply that some changes to the path will be made, which requires that when we go through the arguments layed out before, we will need to rotate while taking these changes into account. This is an instance of where Lemma \ref{lem:technicalrotation} comes in handy, which is the main technical lemma allowing us to rotate in a \emph{robust manner}. In Section \ref{sec:posarot} we discuss the standard P\'osa rotation technique, state and prove our technical lemma as well as discuss our method for closing paths into cycles. Then, in Section \ref{sec:rotthm}, we give the cleaning lemma discussed before (Lemma \ref{lem:findingcleanpartition}) and prove Theorem \ref{thm:main general}.

\subsection{Graphs with many cycles and Theorem \ref{thm:maindense}}
We now discuss the proof of Theorem~\ref{thm:maindense} along with the more general statement (Theorem \ref{thm:pflhamiltonian}) that every 
$(n,d,\lambda)$-graph with many disjoint cycles contains a Hamilton cycle. Indeed, if $d \geq n^{\alpha}$ for some constant $\alpha > 0$, then it is not difficult to show that the graph contains a collection of $\eps n$ many vertex-disjoint cycles, for some constant $\eps$ dependent on $\alpha$. 

For simplicity, let us assume here that all of these cycles are triangles and denote them as $a_ib_ic_i$ for $1 \leq i \leq \eps n$. 
First we will find a path $P_1$ which \emph{connects} linearly many of the triangles, that is, for each such triangle it contains the edge $a_ic_i$ but not the vertex $b_i$ - this can be achieved using rather standard arguments, namely a directed version of the Depth First Search algorithm (see Lemma \ref{lem:DFSaltpath}).
We have now build a \emph{flexible} set $B$ of linear size consisting of all the vertices $b_i$, meaning that we can flexibly use the vertices of $B$ to build structure in the remainder of the graph, and be assured that afterwards, the vertices $b_i$ which we have not used can be ``absorbed'' into the path $P_1$ by replacing $a_ic_i$ with the path $a_ib_ic_i$. 
Moreover, by discarding a few vertices, we can find a set $B'\subseteq B$, still of linear size, such that $G[B']$ has large minimum degree, which will imply that it inherits some expansion properties from $G$, and also such that the endpoints of $P_1$ have large degree into $B'$.

Let $Y$ be the set of vertices not contained in $P_1$ and $B'$. With some additional care, we can ensure that $G[Y]$ has large minimum degree and therefore it has nice expansion properties.
The goal is now to partition $G[Y]$ into few paths, more precisely, at most $o(n/\log n)$ vertex-disjoint paths which cover all vertices of~$Y$, such that the endpoints of all of these paths have neighbours in the flexible set - this is all done in Section~\ref{sec:linforest}. For simplicity, let us start describing briefly how the first property is achieved and then explain how these two properties can be used to find a Hamilton cycle. At the end, we discuss the second property.

One standard argument which has been used in the literature for finding linear forests with few paths in these well-behaved graphs is to find a $2$-factor with few cycles by showing that the total number of $2$-factors is larger than the number of $2$-factors with many cycles using estimates on the permanent as well as facts about (pseudo-)random graphs.
Unfortunately this is not applicable to very sparse graphs and instead we 
need to come up with a different argument. Interestingly, we find a convenient shortcut to this problem. We will use the expanding properties of $G[Y]$ to find in it a spanning regular subgraph and then invoke a result of Alon \cite{alon1988linear} on the linear arboricity conjecture to find the desired collection of paths. Now, let $P_2,\dots,P_r$ be these paths which cover $Y$, with $r = o(n/\log n)$, and let $u_i,v_i$ denote the endpoints of~$P_i$ - assume that each one of these endpoints has neighbours in the flexible set~$B'$. 
We want to now weave the paths together with the path $P_1$ into a Hamilton cycle. More precisely, for each $i$, we want to find a path which connects $v_j$ to $u_{j+1}$ (indices modulo $r$), whose inner vertices lie in $B'$, such that all these paths are vertex-disjoint. Observe that this would indeed yield the desired Hamilton cycle. Since $B'$ enjoys nice expansion properties, we can generally assume to find paths of size $\log n$ between any given pair of vertices. Since we only need to connect $o(n/ \log n)$ pairs, in principal there would be enough space in $B'$ to find all the connections vertex-disjointly. 
This \emph{vertex-disjoint paths problem} has been extensively studied in the literature. We benefit here tremendously from a recent ``roll-back technique'' which allows one to indeed find all these connections in expanding graphs (see~\cite{draganic2022rolling}). Roughly speaking, this technique allows us, for a given pair of vertices we want to connect, to build two binary trees rooted at both vertices in a controlled fashion (using a concept of Friedman and Pippenger~\cite{fp}); then, when these trees are large enough, we are guaranteed by the pseudorandomness of $G$ to find an edge between them which then obviously leads to a path connecting the pair. The astonishing thing is that we can now \emph{roll back}, which means that we can demolish the two trees until only the desired path is left, while having a guarantee that the resulting embedding is still \emph{good} in the sense that we can simply start connecting the next pair disjointly in the same way.
In fact, we will have to use this technique with an additional twist. Namely, in general we cannot find triangles at the start of the proof, but cycles which might be longer. For example, say that we have cycles $a_i,b_i,c_i,d_i$ of length four. Now the path $P_1$ will \emph{connect} some of these in the sense that $P_1$ contains the edges $a_id_i$, but does not contain the vertices $b_i,c_i$. We still have a flexible set now, but we have to use it in such a way that for each $i$, either both $b_i,c_i$ are absorbed (and in a particular way) or neither. In Section~\ref{sec:altpaths}, we deal with this additional complication.

Finally, let us briefly discuss the second property that the endpoints of $P_2,\dots,P_r$ have neighbours in the flexible set $B'$ - indeed, otherwise, we could not connect them using this set. For this, let $X$ be the subset of $Y$ consisting of those vertices which do not have enough neighbours in the flexible set - since $G$ is pseudorandom, we are guaranteed that $X$ is not large. When finding the paths $P_2,\dots,P_r$, we have to make sure that their endpoints avoid~$X$. In order to achieve this, we will first ``absorb'' the whole of $X$ into $Y\setminus X$. This uses a novel idea (see Lemma \ref{lem:coverforest}) on how to cover such a bad set $X$ which might be very useful for future applications (in particular since we do not need the polynomial degree assumption here). Precisely, we can find a collection of disjoint paths which cover $X$ and whose endpoints are in $Y\setminus X$. Then, we will like before use the result of Alon \cite{alon1988linear} to extend this collection into a linear forest with few paths whose endpoints are in $Y\setminus X$, as desired. Clearly, doing this when $X \neq \emptyset$ requires various new ideas.

The part of the paper devoted to Theorem~\ref{thm:maindense} is organized as follows. In Section \ref{sec:altpaths}, we present all the connecting tools we require - both to find at the start the path $P_1$ connecting the cycles as well as to connect the paths $P_2, \ldots, P_r$ at the end using the flexible set to construct the Hamilton cycle. Subsequently, in Section~\ref{sec:linforest} we discuss finding the desired linear forest in our context with few paths with good endpoints. Then, in Section \ref{sec:manycyclesthm} we prove the main theorem.

\section{Preliminaries}\label{sec:prel}
\subsection{Notation}
We mostly use standard graph theoretic notation. Let $G$ be a finite graph. Denote by $V(G)$ its vertex set, and let $S_1,S_2\subseteq V(G)$. We denote by $G[S_1]$ the subgraph of $G$ induced by $S_1$, and by $E[S_1,S_2]$ the set of edges with one endpoint in $S_1$ and the other in $S_2$. For a vertex $x$ in a graph $G$, we let $N_G(x)$ denote the neighbourhood of $x$ in~$G$. For a set $S$ of vertices in a graph $G$, we let $\Gamma_G(S)$ (or just $\Gamma(S)$ if $G$ is clear from the context) denote the neighbourhood of $X$, that is $\Gamma_G(X) := \bigcup_{x \in X} N_G(x)$, and we denote by $N_G(X)$ the external neighbourhood of $X$, that is $N_G(X) = \Gamma_G(X) \setminus X$. If $G$ is a directed graph, then we can also define $N^{+}_G(S), \Gamma^{+}_G(S), N^{-}_G(S), \Gamma^{-}_G(S)$ to be the analogous sets considering out- or in-neighbourhoods. 
Usually $n$ denotes the number of vertices of a given graph, and we assume in all our statements that $n$ is sufficiently large. An event holds \emph{with high probability} if its probability tends to $1$ as $n\to\infty$. Given a collection $\mathcal{M}$ of vertex disjoint sets $S \subseteq V(G)$ in a graph $G$, we will use $V(\mathcal{M})$ to denote the set of all vertices in such sets, that is, $V(\mathcal{M}) := \bigcup_{S \in \mathcal{M}} S$.

\subsection{Some definitions and standard tools}
\noindent We now give two useful definitions concerning pseudorandom graphs.
\begin{defn}
Let $G$ be an $(n,d,\lambda)$-graph. A subset $S \subseteq V(G)$ is said to be \emph{$\delta$-clean} if $G[S]$ has minimum degree at least $\delta$. In the case that $\delta \geq d|S|/4n$, we say simply that $S$ is \emph{clean}.
\end{defn}
\begin{defn}
A graph $G$ is said to be an \emph{$(s,K)$-expander} (or \emph{$(s,K)$-expanding}) if every subset $S \subseteq V(G)$ of size at most $s$ is such that $|N_G(S)| \geq K|S|$. Similarly, a directed graph $G$ is \emph{$(s, K)$-out-expanding} if for every subset $S \subseteq V(G)$ of size $|S| \le s$ we have $|N^{+}_{G}(S)| \ge K |S|$.
\end{defn}
\noindent Now we will state various well-known results that we will use throughout the paper. The first two are standard probabilistic tools.
\begin{lem}[Chernoff's bound - see e.g., \cite{alon2016probabilistic}]\label{lem:chernoff}
Let $X$ be the sum of independent random variables $X_1, \ldots, X_n$ such that $0 \leq X_i \leq k$ for each~$i$. Then, for all $0 < \eps < 1$,
$$\mathbb{P} \left( X < (1-\eps)\mathbb{E}[X]| \right), \mathbb{P} \left( X > (1+\eps)\mathbb{E}[X]|\right) \leq  e^{-\eps^2 \mathbb{E}[X]/3k^2}.$$
\end{lem}
\begin{lem}[Lov\'asz local lemma - see e.g., \cite{alon2016probabilistic}]\label{lem:locallemma}
Let $A_1 ,\ldots , A_n$ be events in a probability
space and suppose that each event $A_i$ is mutually independent of a set of all the other events $A_j$ but at most $d$. If $\mathbb{P}(A_i) \leq p$ for all $i \in [n]$ and $ep(d + 1) \leq 1$, then
$$\mathbb{P} \left(\bigcap_i \overline{A_i} \right) > \left(1-\frac{1}{d+1} \right)^n.$$
\end{lem}
\noindent Next we state the classical result of Hall concerning matchings in bipartite graphs.
\begin{thm}[Hall's theorem]
\label{thm:hall}
Let $G$ be a bipartite graph with bipartition $V(G)=A\cup B$ and such that
$$|N(S)|\ge|S|\quad\forall S\subseteq A.$$
Then $G$ has a matching covering $A$.
\end{thm}
\noindent We also state a simple folklore result to find short cycles in sufficiently dense graphs.
\begin{thm}\label{BSthm} Every $n$-vertex graph with at least $n^{1+1/l}$ edges contains a cycle of length at most $2l$.
\end{thm}
\noindent Finally, we state another folklore result on the spectral gap of $(n,d,\lambda)$-graphs.
\begin{thm}\label{thm:spectralgap}
Every $(n,d,\lambda)$-graph must have $\lambda \geq \sqrt{d \cdot \frac{n-d}{n-1}}$.
\end{thm}
\noindent We remark that by Dirac's theorem stated in the introduction, if $d \geq n/2$ then the graph is Hamiltonian. Hence, we can from now on assume that $d < n/2$ and thus, the above theorem tells us that $\lambda \geq \sqrt{d/2}$. 
\subsection{Expansion properties of pseudorandom graphs}\label{sec:expansion}
Here we state and prove several facts about $(n,d,\lambda)$-graphs which relate to their expansion properties. As it is usual, the starting point is the celebrated expander mixing lemma, introduced by Alon and Chung~\cite{alon1988explicit}, which has since then found numerous applications (the interested reader is referred to Chapter~9 of the book of Alon and Spencer~\cite{alon2016probabilistic}). Below, we state the lemma together with three well-known corollaries. For a proof of these, the reader can consult Section~2 of Krivelevich and Sudakov~\cite{KS:03}. 
\begin{lem}[Expander mixing lemma]\label{lem:expandermixing}
Let $G$ be an $(n,d,\lambda)$-graph. Then, 
\begin{itemize}
    \item[(1)] For every two subsets $A,B \subseteq V(G)$, the number $e(A,B)$ of edges with one endpoint in $A$ and the other in $B$ satisfies
    $$\left|e(A,B) - \frac{d|A||B|}{n} \right| \leq \lambda \sqrt{|A||B|} .$$
    \item[(2)] For every set $A \subseteq V(G)$,
    $$\left|e(A) - \frac{d|A|^2}{2n} \right| \leq \frac{\lambda |A|}{2} .$$
\end{itemize}
Moreover, the following three properties hold.
\begin{itemize}
    \item[(3)] Every $A \subseteq V(G)$ of size at most $\lambda n/d$ has $e(A) \leq \lambda |A|.$
    \item[(4)] For every $A,B \subseteq V(G)$ such that $|A||B| > (\lambda n/d)^2$, there exists an edge with one endpoint in $A$ and the other in $B$.
    \item[(5)] If $d > 2 \lambda$, then $G$ is connected.
\end{itemize}
\end{lem}

\noindent We now give the two crucial lemmas concerning expansion in pseudorandom graphs. These will be corollaries of the expander mixing lemma above. The first states, informally, that as long as there is large enough minimum degree (of order at least $\lambda$), we have good vertex-expansion.
\begin{lem}\label{lem:expansion2}
Let $G$ be an $(n,d,\lambda)$-graph, $G' \subseteq G$ a subgraph and $10 \lambda \leq \delta \leq d$. Suppose that there are subsets of vertices $A, B$ with $|A| \leq \delta n/100d$ and such that every vertex $v \in A$ has $|N_{G'}(v) \cap B| \geq \delta$. Then, $$|N_{G'}(A) \cap B| \geq \min \left(\delta^2|A|/8 \lambda^2, \delta n/10d \right) \geq 10|A|.$$
\end{lem}
\begin{proof}
\noindent Let $T := N_{G'}(A) \cap B$. Firstly, since $|A| \le \delta n/100d$ note that by part \textit{(2)} of Lemma \ref{lem:expandermixing} we have $e_G(A) \leq (\delta/200+\lambda/2 )|A| \leq 0.15\delta|A| $. In turn, by assumption and since $\delta \geq 10 \lambda$, this implies that
$$e_{G}(A,T) \geq e_{G'}(A,T) \geq \delta |A| - 2e_G(A) \geq 0.7 \delta |A| .$$
Now, suppose first that $|A| \leq \lambda^2 n/\delta d$ and for sake of contradiction, also that $|T| < \delta^2|A|/8\lambda^2$. Then, part \textit{(1)} of Lemma \ref{lem:expandermixing} implies that
$$e_{G}(A,T) \leq \frac{d|A||T|}{n} + \lambda \sqrt{|A||T|} < \left(\frac{\delta^2|A|d}{8 \lambda^2 n}  + \lambda \cdot \frac{\delta}{2 \lambda}\right)|A| \leq \left(\frac{\delta}{8}  + \frac{\delta}{2}\right)|A| <0.7 \delta |A|,$$
where the third inequality follows since $|A| \leq \lambda^2 n/\delta d$. This is a contradiction to the previous observation and thus, as desired, $|T| \geq \delta^2|A|/8\lambda^2$. Finally, suppose that $|A| > \lambda^2 n/ \delta d$. By considering a subset $A' \subseteq A$ of size $\lambda^2 n/\delta d$, we have by before that
$|T| \geq  \delta^2|A'|/8\lambda^2 - |A| \geq \delta n/8d - |A| > \delta n/10d$, where we are using that $|A| \le \delta n/100d $.
\end{proof}
\noindent The following is an immediate corollary of the above.
\begin{cor}\label{cor:cleanexpansion}
Let $G$ be an $(n,d, \lambda)$-graph and $S \subseteq V(G)$ be a clean subset of size at least $1000 \lambda n/d$. Then, $G[S]$ is an $\left(|S|/400 ,10\right)$-expander.
\end{cor}
\noindent Next we give another lemma which is a corollary of Lemma \ref{lem:expandermixing}. This will tell us that every sufficiently large set contains a large clean subset. This is important since by Corollary \ref{cor:cleanexpansion}, it implies that it will have good expansion properties. 

\begin{lem}\label{lem:cleaning}
Let $G$ be an $(n,d,\lambda)$-graph and $S \subseteq V(G)$. Then, there are at most $4\lambda^2n^2/d^2|S|$ vertices $u$ such that $|N(u) \cap S| < d|S|/2n$. Moreover, if $|S| \geq 5 \lambda n/d$, there is a clean subset $S' \subseteq S$ of size at least $|S| - \lambda n/d$.
\end{lem}
\begin{proof}
Let $U$ denote the set of vertices $u$ such that $|N(u) \cap S| < d|S|/2n$. Then, $e(U,S) \leq d|U||S|/2n$. At the same time, part \textit{(1)} of Lemma \ref{lem:expandermixing} implies that $e(U,S) \geq d|U||S|/n - \lambda \sqrt{|U||S|}$. Therefore,
$$\lambda \sqrt{|U||S|} \ge d|U||S|/2n, $$
and so, $|U| \leq 4\lambda^2n^2/d^2|S|$ as desired. 

For the second part, do the following process. Start with $S_0 := S$; at each step $ i \geq 1$, consider $G[S_{i-1}]$: if it contains some vertex $u_i$ with degree less than $d|S|/4n$, then remove it, that is, define $S_i = S_{i-1} \setminus \{u_i\}$ and continue; otherwise, stop the process. Now, it is clear that if the process stops at some step $i \leq \lambda n/d$, then the desired subset $S' := S_i \subseteq S$ is produced. For sake of contradiction, suppose then that it only stops later. Consider the set $S_i$ for $i := \lambda n/d + 1$ and define $U := \{u_1, \ldots, u_{i-1}\}$, which is of size at least $\lambda n/d$. Note that $|S_i| = |S| - \lambda n/d \geq \max \left(|S|/2, 4 \lambda n/d \right)$ and further, by definition, every vertex in $U$ must have less than $d|S|/4n < d|S_i|/2n$ neighbours in $S_i$. However, this contradicts the first part of the corollary since $$|U| > \lambda n/d \geq 4 \lambda^2 n^2 /d^2 |S_i|$$
where the second inequality holds since, as noted before, $|S_i| \geq 4 \lambda n/d$.
\end{proof}
\noindent We remark that with the same argument one can easily generalize the second part of the above lemma in the following manner.
\begin{lem}\label{lem:paircleaning}
Let $G$ be an $(n,d,\lambda)$-graph and $\mathcal{S}$ a collection of disjoint pairs $\{x_i,y_i\} \subseteq V(G)$. If $|\mathcal{S}| \geq 3 \lambda n/d$, there is a sub-collection $\mathcal{S}' \subseteq \mathcal{S}$ of size at least $|\mathcal{S}| - \lambda n/d$ such that $V(\mathcal{S}')$ is clean.
\end{lem}

\subsection{Directed graphs and the (directed) Friedman-Pippenger technique}\label{sec:dirgraphs}

\noindent This next section is devoted to stating lemmas concerning finding sparse structures such as paths and trees more generally, as well as connecting structures, in expanding directed graphs, which will often appear in our proofs as auxiliary graphs. We start with a well-known result of Ben-Eliezer, Krivelevich and Sudakov \cite{ben2012size} which gives a sufficient condition for the existence of long paths in directed graphs. Its proof relies on the depth-first-search (DFS) algorithm.
\begin{lem}[\cite{ben2012size}]\label{lem:DFSpath}
Let $H$ be an $n$-vertex digraph and suppose that for all disjoint sets $S,T \subseteq V(H)$ of size $k$ there exists an edge directed from $S$ to $T$. Then, $H$ contains a directed path of length $n - 2k + 1$.
\end{lem}
\noindent We now describe an embedding machinery in expanding directed graphs. This is a directed version of the Friedman-Pippenger technique with roll-backs which is used for embeddings in expander graphs. The technique of Friedman and Pippenger \cite{fp} can be used to embed trees in expanders vertex by vertex, while maintaining a certain invariant. Then, a simple but powerful observation allows one to remove leaves from this tree so that this invariant is still maintained. We refer the reader to \cite{draganic2022rolling} for more on the topic. First, we give the definition of this invariant - a \emph{good embedding}.
\begin{defn}\label{deF:goodness}
Let $H$ be a digraph and let $s, D \in \mathbb{N}$. Given a digraph $F$, we say that an embedding $\phi \colon F \hookrightarrow H$ is \emph{$(s,D)$-good} if
\begin{equation} \label{eq:extendable}
	|\Gamma^{+}_H(X) \setminus \phi(F)| \ge \sum_{v \in X} \left[ D - \deg_F(\phi^{-1}(v)) \right]+|\phi(F)\cap X|
\end{equation}
for every $X \subseteq V(H)$ of size $|X| \le s$. Here we slightly abuse notation by setting $\deg_F(\emptyset) := 0$, i.e. if a vertex $v\in V(H)$ is not used by $\phi$ to embed $F$, then we set $\deg_F(\phi^{-1}(v))=0$.
\end{defn}
\noindent Then, the argument of Friedman-Pippenger (see e.g., Theorem 2.3 and its proof in \cite{draganic2022rolling}) can be adapted straightforwardly towards directed graphs to give the following. 
\begin{thm}\label{FPthm}
Let $F$ be a digraph with $\Delta(F) \le D$ and $v(F) < s$, for some $D, s \in \mathbb{N}$. Suppose we are given a $(2s-2,D{+}2)$-out-expanding digraph $H$ and a $(2s-2,D)$-good embedding $\phi \colon F \hookrightarrow G$. Then for every digraph $F'$ with $v(F') \le s$ and $\Delta(F') \le D$ which can be obtained from $F$ by successively adding a new vertex of in-degree $1$ and out-degree $0$, there exists a $(2s-2,D)$-good embedding $\phi' \colon F' \hookrightarrow G$ which extends $\phi$.
\end{thm}
\noindent Further, the roll-back technique can also be adapted to directed graphs.
\begin{lem}\label{rollbacklem}
Suppose we are given digraphs $G$ and $F$ with $\Delta(F)\leq D$, and an $(s, D)$-good embedding $\phi \colon F \hookrightarrow G$, for some $s, D \in \mathbb{N}$. Then for every digraph $F'$ obtained from $F$ by successively removing a vertex of in-degree $1$ and out-degree $0$, the restriction $\phi'$ of $\phi$ to $F'$ is also $(s, D)$-good.
\end{lem}
\noindent Again, this can be proved by slightly adapting the undirected version - see e.g., Lemma 2.4 in \cite{draganic2022rolling}.
\subsection{Alternating paths in pseudorandom graphs}\label{sec:altpaths}

The purpose of this section is to consider the following setting. We are given a (pseudorandom) graph $G$ and a collection $\mathcal{M}$ of disjoint pairs $\{x,y\} \subseteq V(G)$, with the possibility that $x = y$ - we let $G \cup \mathcal{M}$ denote the result of adding all possible and non-existing edges $xy$ to $G$ when $x \neq y$. We say that a path $P$ in $G \cup \mathcal{M}$ is \emph{$\mathcal{M}$-alternating} if it is of the form $e_1f_1e_2f_2 \ldots e_lf_l$ where the $e_i$'s are edges of $G$ and the $f_i$'s represent either edges $xy$ such that $\{x,y\}$ is a pair in $\mathcal{M}$ or vertices $x$ such that $\{x,x\} \in \mathcal{M}$; informally, without the appearance of pairs $\{x,x\}$, it is just a path in which there is an edge of $\mathcal{M}$ among every two consecutive edges. The goal is now to extend some well-known results concerning paths in pseudorandom graphs to results about $\mathcal{M}$-alternating paths in such graphs. The first will be the following extension of Lemma \ref{lem:DFSpath}.
\begin{lem}\label{lem:DFSaltpath}
Let $G$ be an $n$-vertex graph such that for every two disjoint subsets $S,T \subseteq V(G)$ of size $k$ there is an edge between them. Then, $G \cup \mathcal{M}$ contains an $\mathcal{M}$-alternating path which uses all but $2k-1$ pairs in $\mathcal{M}$. 
\end{lem}
\noindent The second result about alternating paths in pseudorandom graphs will be an extension of the \emph{vertex-disjoint paths problem}. Roughly speaking, this topic deals with the following situation: we are given a graph $G$ and pairs of vertices $\{a_i,b_i\}$; then, we have to find vertex disjoint paths $P_i$ for each $i$ connecting $a_i$ to $b_i$. This problem has been extensively studied and for a detailed discussion on it the reader is referred, e.g., to \cite{draganic2022rolling}. In particular, Dragani\'c, Krivelevich and Nenadov \cite{draganic2022rolling} recently used the Friedman-Pippenger technique to show that this problem can be solved in $(n,d,\lambda)$-graphs (and with a polynomial time online algorithm). In the next result, we extend this result to the `alternating' version that we require.
\begin{thm}\label{thm:vertexdisjointpathsvariation}
Let $G$ be an $(n,d,\lambda)$-graph with $d > 500 \lambda$ and $V(\mathcal{M})$ be $500 \lambda$-clean in $G$ with $|\mathcal{M}| \geq 2000\lambda n/d$. Let $\mathcal{P}$  be a collection of at most $\frac{|\mathcal{M}|}{100 \log_2 n}$ disjoint pairs $\{a_i,b_i\}$ of vertices in $V(G) \setminus V(\mathcal{M})$ such that for all $v \in V(\mathcal{P})$ we have $|N(v) \cap V(\mathcal{M})| \geq 500 \lambda$. Then there exist vertex-disjoint $\mathcal{M}$-alternating paths in $G \cup \mathcal{M}$ between every pair of vertices $\{a_i,b_i\}$.
\end{thm}
\noindent Now, before going into the proofs of both of these results, we define a random auxiliary digraph $H$ for this setting. Indeed, suppose we have a graph $G$ and a collection $\mathcal{M}$ of disjoint pairs $\{x,y
\} \subseteq V(G)$, with the possibility that $x = y$. We define the \emph{random auxiliary digraph $H = H(G,\mathcal{M})$} as follows. For each pair $\{x_i,y_i\} \in \mathcal{M}$, colour uniformly at random one of the vertices \emph{red} and the other \emph{blue} - in the case that $x_i = y_i$ colour the vertex \emph{red} and \emph{blue}. After doing so, re-label these vertices so that $x_i$ always refers to a red vertex and let $X$ denote the set of these; let $Y$ denote the set of blue vertices.  
Now, the vertex set of $H$ is $V(H) := X \cup \big(V(G) \setminus V(\mathcal{M})\big)$ and the edges are defined as follows.
\begin{enumerate}
    \item For any two vertices $x_i, x_j \in X$, put an edge directed from $x_i$ to $x_j$ if $x_iy_j$ is an edge in $G$.
    \item For $v \notin V(\mathcal{M})$ and a vertex $x_i \in X$, put an edge from $v$ to $x_i$ if $vy_i$ is an edge of $G$.
\end{enumerate}
\noindent Crucially, constructing $H$ in such a manner gives the following deterministic property.
\begin{obs}\label{obs:altpath}
If $P = x_{i_1} \rightarrow x_{i_2} \rightarrow \ldots \rightarrow x_{i_l}$ is a directed path in $H$, then $y_{i_1}x_{i_1}y_{i_2}x_{i_2}\ldots y_{i_l}x_{i_l}$ is an $\mathcal{M}$-alternating path in $G \cup \mathcal{M}$. 
\end{obs}
\noindent We can now prove the results stated above. The first one, Lemma \ref{lem:DFSaltpath} will not require the randomness involved in $H$ and follows by showing that the condition of Lemma \ref{lem:DFSpath} is satisfied for $H[X]$ and then applying the observation above.
\begin{proof}[ of Lemma \ref{lem:DFSaltpath}]
Let $S,T$ be two disjoint subsets of $X$ of size $k$. Consider the set $T' := \{y_i: x_i \in T\}$ which has size $k$, and note that by construction, $\overrightarrow{e}_H(S,T) = e_G(S,T')$. By the assumption on $G$, we then have that $\overrightarrow{e}_H(S,T) \neq \emptyset$ and therefore, Lemma \ref{lem:DFSpath} implies that $H[X]$ contains a directed path of length $|X|-2k+1$. By Observation \ref{obs:altpath} this gives an $M$-alternating path as desired.
\end{proof}

\noindent The second result is more technical and will indeed require the randomness of $H$. Precisely, we will first show that with positive probability $H$ satisfies strong expansion properties. Because of this, we are then able to apply the directed Friedman-Pippenger machinery introduced in Section~\ref{sec:dirgraphs}. Before starting, let us state another deterministic property that $H$ satisfies.
\begin{obs}\label{obs:FPaltpath}
Let $P = v \rightarrow x_{i_1} \rightarrow \ldots \rightarrow x_{i_l}$ and $Q = u \rightarrow x_{j_1} \rightarrow \ldots \rightarrow x_{j_t} $ be two disjoint directed paths in $H$ with $u,v \notin V(\mathcal{M})$. Then, if $x_{i_l}x_{j_t}$ is an edge of $G$, we have that $vy_{i_1}x_{i_1}\ldots y_{i_l}x_{i_l}x_{j_t}y_{j_t}\ldots x_{j_1}y_{j_1}u$ is an $\mathcal{M}$-alternating path in $G \cup \mathcal{M}$.
\end{obs}
\begin{proof}[ of Theorem \ref{thm:vertexdisjointpathsvariation}]
Let us fix $\delta := 500 \lambda$. Let us also re-define $H := H[X \cup V(\mathcal{P})]$ - note then that $|V(H)| = |\mathcal{M}| + |V(\mathcal{P})| \leq 2|\mathcal{M}|$. We will first show that $H$ satisfies good expansion properties.
\begin{claim*}
With positive probability, $H$ is $(|V(H)|/10, 5)$-out-expanding.
\end{claim*}
\begin{proof}
We first show that with positive probability, $\delta^{+}(H) \geq \delta/5$ in $H$. Indeed, by assumption initially each vertex $v \in V(\mathcal{P}) \cup V(\mathcal{M})$ of $G$ has at least $\delta$ neighbours in $V(\mathcal{M}) = X \cup Y$. It is easy to check that the random procedure defining $H$ implies that the random variable $|N^{+}_H(v)|$ stochastically dominates the binomial random variable $\text{Bin}(\delta,1/2)$. Indeed, note that if $v$ has both elements of some pair $\{x,y\}$ in $\mathcal{M}$ as neighbours in $G$, then one of them will be an in/out-neighbour in $H$, and if only one of these is a neighbour, then with probability $1/2$ it will be an in/out-neighbour in $H$. Therefore, by Lemma \ref{lem:chernoff} we have that $\mathbb{P}\left(|N^{+}_H(v)| < \delta/5 \mid v \in V(H)\right) = e^{-\Omega(\delta)}$. Moreover, we can further apply Lemma \ref{lem:locallemma}. Indeed, note that clearly since $G$ is $d$-regular, the event $E_v := \{|N^{+}_H(v)| < \delta/5 \mid v \in V(H)\}$ will depend on at most $O(d^2)$ other events $E_{v'}$. Therefore, since $e^{-\Omega(\delta)} = o(d^{-2})$ (by $\delta \geq \sqrt{d/2}$), Lemma \ref{lem:locallemma} implies that with positive probability, no event $E_v$ holds and thus, $\delta^{+}(H) \geq \delta/5$. 

We can now verify the expanding condition. Let $S \subseteq V(H)$ be a set of size at most $|V(H)|/10$ and consider $ X':= \Gamma^{+}_H(S) \subseteq X$. Recall that for each $x_i \in X'$, there is a corresponding $y_i \in Y$ - let $Y'$ denote the set of these. By the paragraph above, we have that for all $v \in S$, $|N^{+}_H(v) \cap X'| \geq \delta/5$. In turn, by the construction of $H$, it must be that for every edge $v \rightarrow x_i$ in $H$ with $x_i \in X'$, there is a corresponding edge $vy_i$ of $G$. Therefore, also $|N_G(v) \cap Y'| \geq \delta/5$. By Lemma \ref{lem:expansion2}, we must then have that if $|S| \leq \delta n/500d$, then $|X'| = |Y'| \geq 10|S|$, as desired. When $\lambda n/d = \delta n/500d < |S| \leq |V(H)|/10$, part \textit{(4)} of Lemma~\ref{lem:expandermixing} applied to $G$ implies that there are at most $\lambda n/d$ vertices of $Y$ which are not adjacent to some vertex of $S$ - that is, $|X'| = |Y'| \geq |Y|-\lambda n/d \geq |V(H)| - \lambda n/d - |V(\mathcal{P})| \geq 0.9|V(H)| \geq 6|S|$, as desired.
\end{proof}
Now that we have proven that $H$ is out-expanding, we can apply the Friedman-Pippenger machinery to find the desired $\mathcal{M}$-alternating paths in $G \cup V(\mathcal{M})$ connecting the pairs in $\mathcal{P}$. Recalling Definition \ref{deF:goodness}, we will from now on refer to a $(|V(H)|/10,3)$-good embedding as only a \emph{good embedding}. Similarly, we will say that a directed graph is \emph{expanding} if it is $(|V(H)|/10, 5)$-out-expanding. 

First note that the set $V(\mathcal{P})$ is a good embedding of an independent set in $H$, since $H$ is expanding and every vertex $v \in V(\mathcal{P})$ has only out-neighbours in $X$, which is disjoint to $V(\mathcal{P})$. We then start by linking the first pair $\{a_1, b_1\} \in \mathcal{P}$. For this, we use Theorem \ref{FPthm} in order to find two disjoint out-oriented binary trees $T_1,T_2$ \emph{in $H$} of size at least $|V(H)|/50 \geq |\mathcal{M}|/50 \geq \lambda n/d$
rooted at $a_1$ and $b_1$ respectively. Furthermore, Theorem \ref{FPthm} guarantees that these trees together with the other vertices in $V(\mathcal{P})$ form a good embedding of an independent set of size $2|\mathcal{P}|-2$ together with two disjoint out-oriented binary trees.  Now, since $|T_1|,|T_2| \geq \lambda n/d$, part \textit{(4)} of Lemma~\ref{lem:expandermixing} implies that there is an edge $e = uw$ in $G[X]$ connecting a vertex $u \in V(T_1)$ and a vertex $w \in V(T_2)$. By Observation \ref{obs:FPaltpath}, this implies the existence of an $\mathcal{M}$-alternating path in $G \cup \mathcal{M}$ with endpoints $a_1, b_1$, which we denote by $P_1$. Note that $P_1$ uses at most $2 \log_2 n$ vertices of $H$.

We can now \emph{roll back} in the following sense. Trivially we can remove vertices from each tree $T_1,T_2$ in a leaf-by-leaf manner so that at the end we are only left with the vertices of $H$ which correspond to the path $P_1$ (which consist of a directed $a_1u$-path and a directed $b_1w$-path). By Lemma \ref{rollbacklem}, we are guaranteed that the current forest remains a good embedding, now of an independent set of size $2|P|-2$ and two disjoint directed paths. We then continue with the same procedure to link the second pair $\{a_2,b_2\} \in \mathcal{P}$ with an $\mathcal{M}$-alternating path $P_2$ disjoint to $P_1$. Again, we first find two complete out-oriented binary trees in $H$ rooted at $a_2$ and $b_2$ (disjoint from $P_1$), then we find an edge in $G[X]$ which connects them and creates the $\mathcal{M}$-alternating $a_2b_2$-path $P_2$, which uses at most $2 \log_2 n$ vertices of $H$. We then remove all the vertices from the trees which do not correspond to $P_2$ in a leaf-by-leaf manner. As such, Lemma \ref{rollbacklem} implies that the resulting forest is a good embedding. We continue this operation for every pair in $\mathcal{P}$; note that we can indeed do this since every $\mathcal{M}$-alternating path we create uses at most $2 \log_2 n$ vertices of $H$ and so, at any given point of the process the current forest to which we are applying Theorem \ref{FPthm} is of size at most $|\mathcal{P}| \cdot 2 \log_2 n + 2 \cdot |V(H)|/50 \leq |\mathcal{M}|/50 + |V(H)|/25 < |V(H)|/10$, where the first term is an upper bound on the total number of vertices of $H$ used in the paths $P_1, P_2, \ldots$ and the second upper bounds the number of vertices in the current binary out-oriented trees which are being used - and thus, this theorem can indeed be applied. This completes the proof since the procedure constructs the desired vertex-disjoint $\mathcal{M}$-alternating paths in $G \cup \mathcal{M}$.
\end{proof}

\section{Robust P\'osa rotation}\label{sec:posarot}
In this section we will give various tools which allow us to close paths into cycles. More specifically, we will say that a path $P$ can be \emph{closed into a cycle} if there exists a cycle on the same vertex set $V(P)$. Our approach to do this heavily relies on the so called rotation-extension technique, invented by P\'osa in \cite{posa:76}, which has subsequently been applied to several problems related to Hamiltonicity (see e.g., \cite{bollobas1987algorithm,frieze2002hamilton,KS:03}). We will in turn develop new tools in order to use this technique in a robust manner. First, we need some definitions and initial lemmas.

\subsection{Standard tools}\label{sec:standardposarot}
We start with some definitions. Let $P = (v_1,v_2, \ldots, v_l)$ be a path in a graph $G$. For a vertex $u = v_i \in V(P)$, we define, as usual, $N_P(u) := \{v_{i-1}, v_{i+1}\}$ to be the set of neighbours of $u$ in $P$. When necessary and when such an ordering of the path is defined, we will use $u^{+}_P := v_{i+1}$ to denote its right neighbour in $P$ and $u^{-}_P := v_{i-1}$ to denote its left neighbour. Given a subset $X \subseteq V(P)$, we define $\text{int}_P(X)$ (or just $\text{int}(X)$ when $P$ is clear from the context) to be the set of vertices $v_i \in X$ with $i \notin \{1,l\}$ such that $v_{i-1},v_{i+1} \in X$.
\begin{defn}\label{def:differencebetweenpaths}
For two paths $P,P'$ on the same vertex set $V$, we define the \emph{difference between $P$ and $P'$}, denoted as $\text{dif}(P,P')$, to be the set of vertices $u \in V$ such that $N_P(u) \neq N_{P'}(u)$.  
\end{defn}
\noindent It is worth making some remarks about the above definition. First, note that clearly the set $V \setminus \text{dif}(P,P')$ spans the same collection of vertex-disjoint paths in both $P$ and $P'$ - moreover, there are at most $|\text{dif}(P,P')|+1$ of these paths. Note also that for any subset $X \subseteq V \setminus \text{dif}(P,P')$, we have $\text{int}_P(X) = \text{int}_{P'}(X)$. 

We can now start introducing the influential rotation-extension technique of P\'osa.
\begin{defn}\label{def:rotation}
If $1 < i < l$ and $v_i v_l$ is an edge of $G$, the path $$P' = (v_1,v_2, \ldots, v_i, v_l, v_{l-1}, \ldots, v_{i+1})$$
is said to be a \emph{rotation of} $P$ with \emph{fixed endpoint} $v_1$, \emph{pivot} $v_i$ and \emph{broken edge} $v_i v_{i+1}$. More generally, given a subset $X \subseteq V(P)$, we will say that a path $P'$ is an \emph{$(X,t)$-rotation of $P$} with \emph{fixed endpoint $v_1$} if it is the result of at most $t$ consecutive rotations starting with $P$, each with fixed endpoint $v_1$, broken edges in $E(P)$ and pivots belonging to the set $\text{int}_P(X)$. Clearly, such a path $P'$ has as endpoints the vertex $v_1$ and a vertex in $X$. We say that it is just a \emph{$t$-rotation} (or respectively, an \emph{$X$-rotation}) if there is no restriction on where the pivots and broken edges belong to (or respectively, on the number of rotations). 
\end{defn}
\noindent Again, we make two remarks about the above definition. 
\begin{obs}\label{obs:rotation}
Let $P'$ be a $1$-rotation of $P$ as above. Then, only the vertices $v_i,v_{i+1},v_l$ change neighbourhoods, that is, $N_{P}(u) \neq N_{P'}(u) \Rightarrow u \in \{v_i,v_{i+1},v_l\}$. Moreover, 
\begin{itemize}
    \item For all $X \subseteq V(P)$ such that $v_l \notin X$, then either the new endpoint $v_{i+1}$ of $P'$ belongs to $X$ or $ \mathrm{int}_{P'}(X) = \mathrm{int}_{P}(X)$.
    \item If $P'$ is a $t$-rotation of $P$, then $|\mathrm{dif}(P,P')| \leq 3t$.
\end{itemize}
\end{obs}
\noindent Now, we give two lemmas. The first is an analogue of Lemma \ref{lem:cleaning} about finding clean subsets in the context of paths and P\'osa rotation. Let us first define such a subset. 
\begin{defn}
Let $G$ be an $(n,d,\lambda)$-graph and $P$ a path in $G$. A subset $S \subseteq V(P)$ is said to be \emph{$(P,\delta)$-clean} if for all $v \in S$ we have that $|N(v) \cap \text{int}_P(S)| \geq \delta$. We say only that $S$ is \emph{$P$-clean} if it is \emph{$(P,d|S|/4n)$-clean}.
\end{defn}
\noindent Given this, we now have the following analogous statement to Lemma \ref{lem:cleaning}.
\begin{lem}\label{lem:cleaningrotation}
Let $G$ be an $(n,d,\lambda)$-graph, $P$ a path and $A$ a subpath of $P$. If $|A| \geq 10\lambda n/d$, there is a $P$-clean subset $A' \subseteq A$ with $|A'| \geq 0.9|A|$.
\end{lem}
\begin{proof}
Similarly to the proof of Lemma \ref{lem:cleaning}, we do the following process. Start with $A_0 := A$ and at each step $i \geq 1$, remove a vertex $u_i \in A_{i-1}$ such that $|N(u_i) \cap \text{int}_{P}(A_{i-1})| < d|A|/4n$, setting $A_i := A_{i-1} \setminus \{u_i\}$; if no such vertex exists then stop the process. Clearly, if the process stops at some step $i \leq 0.1|A|$, then the desired subset $A' := A_i$ is constructed. Suppose then, for sake of contradiction, that the process only stops after that and consider the set $A_i$ with $i := 0.1|A| + 1$. Define also the set $U := \{u_1, \ldots, u_{i-1}\}$ which has size $0.1|A|$ and note that by construction, we must have that every vertex in $U$ has at most $d|A|/4n$ neighbours in $\text{int}_{P}(A_{i-1})$. In turn, we claim that this contradicts the first part of Lemma \ref{lem:cleaning} with $S := \text{int}_{P}(A_{i-1})$. Indeed, note that since $A$ is a subpath of $P$ and we have $|A_{i-1}| \geq 0.9|A|$, it must be that $|S| = |\text{int}_{P}(A_{i-1})| \geq 0.7|A|$. In particular, $d|A|/4n < d|S|/2n$ and so, the first part of Lemma~\ref{lem:cleaning} gives that
$$|A| \cdot |U| \le |A| \cdot \left(4 \lambda^2 n^2/d^2 |S| \right) \leq 8 \lambda^2 n^2/d^2$$
which is a contradiction since $|A| \cdot |U| \geq 0.1 |A|^2 \geq 10 \lambda^2 n^2/d^2$ by assumption.
\end{proof}
The next lemma is a standard version of the usual P\'osa rotation lemma which has been implicitly used for several Hamiltonicity problems. Its proof can be found, e.g., in \cite{KS:03}.
\begin{lem} [\cite{KS:03}]\label{lem:usualrotation1}
Let $G$ be an $(n,d, \lambda)$-graph with $d \geq 10 \lambda$, $P$ a maximal path in $G$ with endpoints $x,y$ and $|P| \geq n/10$. Let $Z_1 \subseteq Z_2 \subseteq \ldots \subseteq V(P)$ be a nested sequence defined as follows: for each $i \geq 1$, $Z_i$ is the set of vertices $z$ for which there exists an $i$-rotation of $P$ with fixed endpoint $y$ which is a $zy$-path. Then, $Z_1 \neq \emptyset$ and $|Z_{i+1}| \geq \min \left(2|Z_i| ,n/100\right)$ for all $ i \geq 1$.
\end{lem}
\noindent We remark that a proof of this will be implicit in the more general Lemma~\ref{lem:technicalrotation} in the next section. Crucially, the above lemma implies the following.
\begin{cor}\label{cor:usualrotation}
Let $G$ be an $(n,d, \lambda)$-graph with $d \geq 10 \lambda$, $P$ a maximal path in $G$ with endpoints $x,y$ and $|P| \geq n/10$. Let $R \subseteq V(P)$ be a set such that $|\mathrm{int}_P(R)| \geq 100 \lambda^2 n/d^2$. Then, there exists a $(\log n)$-rotation $P'$ of $P$ with $y$ as one endpoint and the other endpoint $x' \in R$ such that $\left| \mathrm{int}_P(R) \setminus \mathrm{int}_{P'}(R) \right| \leq 3$.
\end{cor}
\begin{proof}
Let us apply Lemma \ref{lem:usualrotation1} and take $I$ to be the minimal $i$ such that $|Z_i| \geq n/100$ - by the lemma, we have that $I \leq \log n - 1$. If $R \cap Z_I \neq \emptyset$, then let $x' \in R \cap Z_I$ belong to $Z_j$ with $j$ minimal. Then, we claim that $x'$ satisfies the desired property. Indeed, it is such that there is a $j$-rotation $P'$ of $P$ with $y$ as one endpoint and the other endpoint $x'$. Moreover, since $j$ is minimal note that the endpoints created by each one of the previous rotations forming $P'$ are not in $R$ except for $x'$. Thus, by the first item of Observation \ref{obs:rotation}, the only change to $\text{int}_P(R)$ occurs in the last step and so, by the same observation, it only affects at most three vertices, hence $\left| \text{int}_P(R) \setminus \text{int}_{P'}(R) \right| \leq 3$, as desired. 

Suppose now that $R \cap Z_I = \emptyset$. Since $|Z_I||\text{int}_P(R)| \geq \lambda^2 n^2/d^2$, part \textit{(4)} of Lemma \ref{lem:expandermixing} implies that there is an edge $zv$ with $z \in Z_I$ and $v \in \text{int}_P(R)$. Now, consider the $zy$-path $P_z$ which is an $I$-rotation of $P$ with fixed endpoint $y$. Since $R \cap Z_I = \emptyset$, Observation \ref{obs:rotation} implies that $ \text{int}_P(R) =  \text{int}_{P_z}(R) $. Therefore, $v \in \text{int}_{P_z}(R)$ and so, the edge $zv$ ensures that there is a path $P'$ which is a rotation of $P_z$ with fixed endpoint $y$ and an $x'y$-path for some $x' \in R$. We now check that $P',x'$ satisfy the desired conditions. Indeed, $P'$ is an $(I+1)$-rotation of $P$ with $I+1 \leq \log n$. Further, since $ \text{int}_P(R) =  \text{int}_{P_z}(R) $ and $P'$ is a $1$-rotation of $P_z$ we have that $\left| \text{int}_P(R) \setminus \text{int}_{P'}(R) \right| \leq 3$.
\end{proof}
\subsection{The main technical rotation lemma}
\noindent We will further need a much more robust extension of the P\'osa rotation technique than Lemma \ref{lem:usualrotation1}. Informally, given a path $P$ we will first want to perform rotations only in a certain \emph{clean} subset of $P$ and secondly, we want to be able to do this robustly, that is, possibly after some changes to the path have been done. Clearly, the challenge here is that these changes might alter the \emph{clean} subset we wish to rotate in.
\begin{lem}\label{lem:technicalrotation}
Let $G$ be an $(n,d,\lambda)$-graph with $d \geq 100 \lambda$, $P$ a path in $G$, $X,Y \subseteq V(P)$ be non-empty with $X$ being a $(P,\delta)$-clean subset for some $20 \lambda \leq \delta \leq d$. Let also $P'$ be another path with $V(P') = V(P)$ and let $x,y$ be its endpoints with $x \in X$. Suppose the following holds in the case that $P \neq P'$.
\begin{itemize}
    \item At most $\delta/2$ vertices $w \in \mathrm{int}_P(X) \setminus \mathrm{int}_{P'}(X)$ are such that $\mathrm{dist}_G(w,x) \leq 2\log |\mathrm{dif}(P,P')|$.
\end{itemize}
Then there are at least $\delta n/200d$ vertices $z \in X$ for which there exists a $zy$-path $P_z$ which is an  $(X,\log  n)$-rotation of $P'$ with fixed endpoint $y$ and such that $|\mathrm{dif}(P_z,P') \cap Y| \leq 3\log |Y|$.
\end{lem}
\begin{proof}
Define $F := \text{dif}(P,P')$, so that we have for all $u \notin F$ that $N_P(u) = N_{P'}(u)$, and as observed before, $\text{int}_P(X \setminus F) = \text{int}_{P'}(X \setminus F)$. Moreover, clearly we have $|\text{int}_P(X \setminus F)| \geq |\text{int}_P(X)| - 3|F|$ which implies the following: for every $S \subseteq X$ we have $|N(S) \cap \text{int}_{P'}(X)| \geq |N(S) \cap \text{int}_{P'}(X \setminus F)| \geq |N(S) \cap \text{int}_P(X)| - 3|F|$.
\noindent Now, for each $i \geq 0$, let us define $Z_i \subseteq X$ to be the set of vertices $z$ for which there exists an $(X,i)$-rotation $P_z$ of $P'$ with fixed endpoint $y$ which is a $zy$-path and such that $|\text{dif}(P_z,P') \cap Y| \leq 3\log |Y|$. Clearly, note that $\{Z_i\}_i$ is a family of nested sets with $Z_i \subseteq Z_{i+1}$ for each $i$. Further, the following recurrence holds.
\begin{claim*}
For each $i < \log |Y|$, we have $|Z_{i+1}| \geq \frac{1}{2}|N(Z_i) \cap \mathrm{int}_{P'}(X)| - \frac{3}{2}|Z_i| .$
For $i \geq \log |Y|$, we have
$$|Z_{i+1}| \geq \frac{1}{2} \left|N(Z_i) \cap \mathrm{int}_{P'}(X)\right| - \frac{3}{2}|Z_i|  - \frac{3}{2}|Y|. $$
\end{claim*}
\begin{proof}
Let us suppose first that $i < \log |Y|$. Fix an initial ordering $(v_1 := x,v_2, \ldots, v_{l-1}, v_l := y)$ of the path $P'$ and define the set of indices $J := \{1< j < l : v_j \in N(Z_i) \cap \text{int}_{P'}(X), v_{j-1},v_j,v_{j+1} \notin Z_i\}$. Clearly we have that $|J| \geq |N(Z_i) \cap \text{int}_{P'}(X)| - 3|Z_i|$. Consider a vertex $v_j$ with $j \in J$. We claim that one of $v_{j-1},v_{j+1}$ belongs to $Z_{i+1}$, thus showing that $|Z_{i+1}| \geq |J|/2$, which by the previous lower bound on $|J|$ implies the desired recurrence. Indeed, note that since $v_j \in N(Z_i) \cap \text{int}_{P'}(X)$, it has a neighbour $z \in Z_i$. By definition, there exists an $(X,i)$-rotation of $P'$ with fixed endpoint $y$ and whose other endpoint is $z$ - let us denote this path by $Q$. Note that the subpath $v_{j-1}v_jv_{j+1}$ is still present in $Q$ (albeit in maybe a different direction). Indeed, during the repeated rotation process of obtaining $Q$, by definition every broken edge must be incident to some vertex in $Z_i$ and we have $v_{j-1},v_j,v_{j+1} \notin Z_i$. Therefore, we can use the edge $zv_j$ to perform a rotation of $Q$ with fixed endpoint $y$, pivot $v_j$ and broken edge either $v_{j-1}v_j$ or $v_jv_{j+1}$. This results in an $(X,1)$-rotation of $Q$ and so, an $(X,i+1)$-rotation of $P'$ with fixed endpoint $y$. The other endpoint of the path will be either $v_{j-1},v_{j+1}$ and so one of these two vertices belongs to $Z_{i+1}$. Note further that this path, which we denote as $Q'$, an $(X,i+1)$-rotation of $P'$ with fixed endpoint $y$, by Observation \ref{obs:rotation} is such that $|\text{dif}(Q',P') \cap Y| \leq |\text{dif}(Q',P')| \leq 3i+3 \leq 3\log |Y|$.

For the case that $i \geq \log |Y|$, the proof is analogous, the only difference being that we define $J := \{1< j < l : v_j \in N(Z_i) \cap \text{int}_{P'}(X), v_{j-1},v_j,v_{j+1} \notin Z_i \cup Y\}$, which has size at least $|N(Z_i) \cap \text{int}_{P'}(X)| - 3|Z_i| - 3|Y|$. We can then show similarly as above that for each $v_j \in J$, one of $v_{j-1},v_{j+1}$ belongs to $Z_{i+1}$ - and thus, $ |Z_{i+1}| \geq |J|/2$. Indeed, there is a $z \in Z_i$ which is a neighbour of $v_j$, so that there is an $(X,i)$-rotation $Q$ of $P'$ with fixed endpoint $y$ and whose other endpoint is $z$. The edge $zv_j$ can then be used to extend this to an $(X,i+1)$-rotation of $P'$ with fixed endpoint $y$ and one of $v_{j-1},v_{j+1}$ as the other endpoint - without loss of generality, assume that it is $v_{j-1}$ and let this path be denoted as $Q'$. Note that $\text{dif}(Q',P') \setminus \text{dif}(Q,P') \subseteq \{v_j,v_{j-1}\}$ and so, since $v_{j-1},v_j,v_{j+1} \notin Y$, we have by definition that $|\text{dif}(Q',P') \cap Y| \leq |\text{dif}(Q,P') \cap Y| \leq 3\log |Y|$ as desired.
\end{proof}
\noindent Now, note that by assumption we have $|Z_1| = |N(x) \cap \text{int}_{P'}(X)| \geq \delta/2$. Indeed, since $X$ is $(P,\delta)$-clean, we have that $|N(x) \cap \text{int}_{P}(X)| \geq \delta$ - therefore, if $P = P'$ this trivially holds and if $P \neq P'$, then the assumption of the lemma implies that then $|N(x) \cap \text{int}_{P'}(X)| \geq \delta/2$ (since any neighbour $w \in N(x)$ clearly has $\text{dist}(w,x) \leq 2 \log |F|$). Thus, in order to finish the proof of the lemma we only need to show that while $|Z_i| \leq \delta n/200d$, we have that $|N(Z_i) \cap \text{int}_{P'}(X)| \geq 10|Z_i|$. Indeed, if that is the case, then by the claim above we have first for each such $\log |Y| > i \geq 1$ that $|Z_{i+1}| \geq 5|Z_i| - \frac{3}{2}|Z_i| \geq 2|Z_i|$. In particular, this implies that $|Z_I| \geq 2^{I} \geq |Y|$ for $I := \log |Y|$. Then, again by the claim above we have that for all $i \geq I$, $|Z_{i+1}| \geq 5|Z_i| - \frac{3}{2}|Z_i| - \frac{3}{2}|Y| \geq 5|Z_i| - 3|Z_i| \geq 2|Z_i|$. Therefore, $|Z_l| \geq \delta n/200d$ will occur for some $l \leq \log n$. This in turn gives us the desired vertices $z$ - those contained in $Z_l$. 

Now, suppose then that $i$ is such that $|Z_i| \leq \delta n/200d$ and consider $N(Z_i) \cap \text{int}_{P'}(X)$. Note first that when $P \neq P'$ and $i < \log |F|$, since P\'osa rotation implies that $N(Z_i)$ is composed of vertices which are at a distance at most $2i + 1 < 2\log |F|$ of $x$, we have by assumption of the lemma that every vertex in $Z_i$ has at least $\delta/2$ neighbours in $\text{int}_{P'}(X)$ (recall that every such vertex has at least $\delta$ neighbours in $\text{int}_{P}(X)$). Applying now Lemma \ref{lem:expansion2} with $G' := G$, $A:= Z_i$, $B := \text{int}_{P'}(X)$ we get that $|N(Z_i) \cap \text{int}_{P'}(X)| \geq 10|Z_i|$. Now, let us consider when $i \geq \log |F|$ or $P = P'$ (then clearly $F = \emptyset$). Note that the previous observations imply that $|Z_i| \geq 2^{\log |F|} \geq |F|$ - hence, $|N(Z_i) \cap \text{int}_{P'}(X)| \geq |N(Z_i) \cap \text{int}_{P}(X)|  - 3|Z_i|$ (note that this also holds if $P = P'$). Since every vertex in $Z_i$ has at least $\delta$ neighbours in $\text{int}_{P}(X)$, Lemma \ref{lem:expansion2} then implies also that $|N(Z_i) \cap \text{int}_{P}(X)| \geq \min \left( \frac{\delta^2}{8 \lambda^2},20\right)|Z_i| \geq 13 |Z_i| $. Hence, $|N(Z_i) \cap \text{int}_{P'}(X)| \geq 13|Z_i|  - 3|Z_i| \geq 10|Z_i|$, as desired. 
\end{proof}
\noindent Note that as a corollary, we get the following which tells us that we can always rotate inside $P$-clean subsets and avoid certain vertices.
\begin{cor}\label{cor:rotationinset}
Let $G$ be an $(n,d,\lambda)$-graph with $d \geq 100 \lambda$, $P$ a path in $G$ with endpoints $x,y$ and $X, Y \subseteq V(P)$ be such that $X$ is a $(P,\delta)$-clean subset for some $20 \lambda \leq \delta \leq d$ and $x \in X$. Then there are at least $\delta n/200d$ vertices $z \in X$ for which there exists a $zy$-path $P_z$ which is an  $(X,\log n)$-rotation of $P$ with fixed endpoint $y$ and such that $|\mathrm{dif}(P_z,P) \cap Y| \leq 3\log |Y|$.
\end{cor}
\subsection{A method for closing paths into cycles}
We finish with a general setup which will always allow one to close a given path into a cycle. This was already alluded to in Section \ref{sec:outline}. The challenges in achieving this were also discussed there and will be dealt with in the next section. Before stating the setting, let us give a definition. For a path $P$ and a subset $A \subseteq V(P)$, the graph induced by $P$ on $A$, that is, $P[A]$, is a disjoint union of paths - we will let $c_P(A) := c \left(P[A] \right)$ denote the number of such paths, that is, the number of components of $P[A]$.

\begin{lem}\label{lem:closingpathintocycle}
Let $G$ be an $(n,d,\lambda)$-graph, $P$ a path in $G$ with endpoints $x,y$ and $A,B \subseteq V(P)$ be disjoint sets such that the following holds. Both $A$ and $B$ contain $\left(P,200\lambda \cdot 2^{\min (c_{P}(A),c_{P}(B))} \right)$-clean subsets $A' \subseteq A, B' \subseteq B$ with $x \in A'$ and $y \in B'$. Then, $P$ can be closed into a cycle.
\end{lem}
\begin{proof}
Let us fix an ordering $P = (x,v_2, \ldots, v_{l-1},y)$ and suppose without loss of generality that $c_{P}(A) \geq c_{P}(B)$. First, we apply Corollary~\ref{cor:rotationinset} with $X := A', Y := \emptyset$ thus finding at least $( \lambda n/d) \cdot 2^{c_P(B)}$ vertices $z$ such that there is a $zy$-path $P_z$ which is an $A'$-rotation of $P$ with fixed endpoint $y$ - let $Z_1$ denote the set of these vertices. Note crucially that for each $z \in Z_1$, since $P_z$ is an $A'$-rotation of $P$, we have that $P_z[B] = P[B]$ and so, $c_{P_z}(B) = c_P(B)$. Moreover, the following holds.
\begin{claim*}
There exists a subset $Z'_1 \subseteq Z_1$ of size at least $|Z_1| \cdot 2^{-c_P(B)}$ such that the following holds: for all $v \in B$ and $z,z' \in Z'_1$ we have $v^{+}_{P_z} = v^{+}_{P_{z'}}$ and $v^{-}_{P_z} = v^{-}_{P_{z'}}$.
\end{claim*}
\begin{proof}
For the purpose of the argument let us fix $c := c_P(B)$. As noted above, we have for each $z \in Z_1$ that $P_z[B] = P[B]$. Indeed, the only difference between $B$ in $P_z$ and in $P$ is the direction in which each path of $P_z[B] = P[B]$ is traversed. More precisely, let us recall the ordering of $P$ as $(x, \ldots, y)$, that is, from $x$ to $y$ and also consider the ordering of $P_z$ from $z$ to $y$. Let $Q$ be one of the $c$ sub-paths of $P$ which form $B$, let $u,v$ be its endpoints and assume that in the given ordering of $P$ we have that $Q$ is ordered from $u$ to $v$. We know that $Q$ is also a sub-path of $P_z$, but we are unaware of which direction $Q$ is ordered in the ordering of $P_z$ from $z$ to $y$, there being $2$ possibilities ($u \rightarrow v$ or $v \rightarrow u$). Furthermore, this direction will determine $v^{+}_{P_z}$ and $v^{-}_{P_z}$ for all $v \in Q$. Therefore, by pigeonholing, since there are $c$ sub-paths $Q$ we can find a fraction of $2^{-c}$ of the vertices $z \in Z_1$ for which all of these are equal, as desired.  
\end{proof}
\noindent Note that the above claim implies that if a vertex $w \in B$ is such that for some $z \in Z'_1$ there is a $zw$-path which is a $B'$-rotation of $P_z$ with fixed endpoint $w$, then this is the case for all other vertices $z' \in Z'_1$. Therefore, letting $W \subseteq B$ denote the set of such vertices $w$ we have that for all $z \in Z'_1, w \in W$, there is a $zw$-path on the vertex set $V(P)$. Hence, if there exists an edge $zw$ between two such vertices, then $P$ can be closed into a cycle. In order to find such an edge, we only need to show that $|W| \geq \lambda n/d$ since then $|W||Z'_1| \geq (\lambda n/d)|Z_1| \cdot 2^{-c_P(B)} \geq (\lambda n/d)^2 $ and thus, part \textit{(4)} of Lemma \ref{lem:expandermixing} guarantees that the edge exists. To show that $|W| \geq \lambda n/d$, fix some $z \in Z'_1$ and note that $B'$ is a $(P_z,200\lambda)$-clean subset of $P_z$ and $y \in B'$. So, Corollary \ref{cor:rotationinset} applied on the path $P_z$ with $X:= B'$ implies that $|W| \geq \lambda n/d$.
\end{proof}

\section{Hamilton cycles with robust P\'osa rotation}\label{sec:rotthm}
In this section we will prove Theorem~\ref{thm:main general}. In order to do this, we will show that it will always be possible to close a maximal path into a cycle. Clearly, if a graph $G$ is connected (as are $(n,d,\lambda)$-graphs with $\lambda < d/2$ by part \textit{(5)} of Lemma~\ref{lem:expandermixing}) and every maximal path can be closed into a cycle, then it is Hamiltonian. As expected, our arguments here will build on the tools given in the previous section. First, we introduce the notion of a \emph{clean} collection of intervals of a path. This will be crucial for then closing this path into a cycle. 
\begin{defn}\label{def:cleanedcollection}
Given a path $P$, a collection $\mathcal{Q} = \{Q_1, \ldots, Q_k\}$ of disjoint sub-paths of $P$ is said to be $(\delta,\gamma,k)$-\emph{clean} for $P$ if it is equipped with a subset $S \subseteq \bigcup_i Q_i$ and has the the following properties.
\begin{enumerate}
    \item $|Q_i| \ge 0.99|P| /k$  and $|S \cap Q_i| \geq 0.99|Q_i|$ for each $i$.
    \item Every vertex $v \in S$ has at least $\delta$ neighbours in $\text{int}_P(S)$ - that is, $S$ is $(P,\delta)$-clean.
    \item For each vertex $v \in S$ there are at least $\gamma k$ indices $j$ such that $|N(v) \cap \text{int}(S \cap Q_j)| \geq \delta/k$.
\end{enumerate}
\end{defn}
\noindent We remark first that the reader should think of the above definition with $\delta$ being less than $d|S|/20n$, rather than $d|S|/4n$, where the second condition would tell us that $S$ is $P$-clean. This then makes the third condition much more plausible to obtain. We remark also that this notion appeared in \cite{knox2012approximate} in the context of random graphs in a much more simplified form. We can now show that every sufficiently large path in a pseudorandom graph contains a clean collection - we remark that in applications, the value of $k$ in the lemma below will be small, of polylogarithmic order in terms of $n$.

\begin{lem}\label{lem:findingcleanpartition}
Let $\alpha := 10^{-8}$ and $G$ be an $(n,d,\lambda)$-graph with $\lambda \leq \alpha d$ and $P$ a path in $G$ of size at least $n/3$. Then for all $k$ there is a $\left(d/100,\frac{\alpha d^2}{k \lambda^2},k \right)$-clean collection for $P$.
\end{lem}
\begin{proof}
Let $\delta = d/100$, $\beta := 10^{-5}$ and $\gamma := \frac{\alpha d^2}{k \lambda^2}$. Let us first partition $P$ into disjoint subpaths $I_1, \ldots, I_r$ each of size $ |P|/r$ with $r := 1.01 k$. Initially take $S := V(P)$ and iteratively do the following procedure until it is no longer possible: (\textit{1}) If there exists an index $j$ such that $|S \cap I_j| < 0.99|I_j|$, then remove every vertex of $I_j$ from $S$ and continue; (\textit{2}) If that is not the case and there is a vertex $v \in S$ with degree less than $\delta$ in $\text{int}(S)$, remove it from $S$ and continue; (\textit{3}) If that is also not the case and there is a vertex $v \in S$ such that there are less than $\gamma r$ indices $j$ for which $|N(v) \cap \text{int}(S \cap I_j)| \geq \delta/k$, remove it from $S$ and continue.

The goal is then to show that this process must stop while the number of indices $j$ such that $|S \cap I_j| \geq 0.99|I_j|$ is at least $k$ - we then take $\mathcal{Q}$ to be the collection of these sub-paths and it will have the desired properties. Suppose indeed that this is not the case. Then, we can assume that we achieve a stage of the process when the number of indices $j$ such that $|S \cap I_j| \geq 0.99|I_j|$ is precisely $k$ - let $\mathcal{I}$ be the set of such indices. Define also $R \subseteq V(P) \setminus S$ to be the set of vertices which have been previously \emph{removed} in the process because one of the options (\textit{2}) or (\textit{3}) described before occurred. Observe first then that the following must hold.
\begin{claim*}
$|R| \geq (r-k) \cdot 0.01|P|/r \geq \beta n$.
\end{claim*}
\begin{proof}
By the definition of the process, it must be that every subpath $I_j$ that has been removed from $S$ (that is, using step \textit{(1)}) is such that $0.01|I_j| = 0.01|P|/r$ of its vertices were previously removed individually and thus, are contained in $R$. Since by assumption exactly $r-k$ subpaths $I_j$ have been removed from $S$, we have that $|R| \geq (r-k) \cdot 0.01|P|/r \geq \beta n$.
\end{proof}
\noindent Moreover, by definition every vertex $v \in R$ satisfies one of the following properties.
\begin{itemize}
    \item $v$ has degree less than $\delta$ in $\text{int}(S)$.
    \item There are fewer than $\gamma r$ indices $j \in \mathcal{I}$ for which $|N(v) \cap \text{int}(S \cap I_j)| \geq \delta/k$. 
\end{itemize}
\noindent Let $R_1, R_2 \subseteq R$ respectively denote the sets of vertices of $R$ satisfying these properties. We will now give upper bounds for both $R_1,R_2$ using Lemma \ref{lem:expandermixing} and thus achieve a contradiction to the above claim. First, consider $R_1$. Note that since $|S| \geq \sum_{j \in \mathcal{I}} |S \cap I_j| \geq \sum_{j \in \mathcal{I}} 0.99|I_j| \geq k \cdot 0.99|P|/r \geq 0.98|P|$, we have that $|\text{int}(S)| \geq |S| - 2|P \setminus S| \geq 0.9|P| \geq n/4$. Therefore, $\delta \leq d|\text{int}(S)|/2n$ and so, by Lemma \ref{lem:cleaning}, we must have $|R_1| \leq 4 \lambda^2 n^2/d^2 |\text{int}(S)| \leq \beta n/3 $. 

Now, consider $R_2$. We claim that it has size at most $\beta n/3$, which is a contradiction since then $\beta n \le |R| \le |R_1| + |R_2| \leq 2 \beta n/3$. Indeed, define an auxiliary bipartite graph $H$ with one part $A$ corresponding to $R_2$ and the other part $B$ corresponding to $\mathcal{I}$ - put an edge in $H$ between $v \in A$ and $j \in B$ if $|N(v) \cap \text{int}(S \cap I_j)| < \delta/k$. By definition of $R_2$, every vertex in $A$ has more than $k-\gamma r \geq (1-2\gamma)k = (1-2\gamma)|B|$ neighbours in $B$ and so, we can greedily find a subset $B' \subseteq B$ of $\frac{1}{4\gamma}$ vertices with at least $|A|/2$ common neighbours in $A$. Let $A'$ denote this set of common neighbours and let $X := \bigcup_{j \in B'} \text{int}(S \cap I_j)$. Recall that for each $j \in B'$ we have $|S \cap I_j| \geq 0.99|I_j|$ and so, $\text{int}(S \cap I_j) \geq |S \cap I_j| - 2|I_j \setminus S| \geq  0.97|I_j| \geq 0.9|P|/r \geq n/5r$ - hence, $|X| \geq |B'|n/5r \geq n/20\gamma r$. Further, by definition of $H$ we must have that every vertex $v \in A'$ has at most $\delta|B'|/k = \delta/4 \gamma k < d|X|/2n$ neighbours in $X$. So, by Lemma \ref{lem:cleaning} we must have that $|R_2| \leq 2|A'| \leq 8 \lambda^2 n^2/d^2 |X| <\beta n/3$.

\end{proof}

\subsection{Proof of Theorem \ref{thm:main general}}
\noindent We will now prove the first main theorem of the paper. Let $G$ be an $(n,d,\lambda)$-graph with $d/\lambda \geq C(\log n)^{1/3}$ where $C$ is some large absolute constant. Fix $\delta = d/100$, $\gamma := \frac{C}{(\log n)^{1/3}} \geq C^2\lambda/d$, $k := 30 \log n$ and $D = d/(\log n)^{1/3}$. Let $P$ be a path of maximal length in $G$ - we will show that $P$ can be closed into a cycle. Since $G$ is connected this clearly implies that $G$ is Hamiltonian. Note also that an undirected version of Lemma~\ref{lem:DFSpath} along with part \textit{(4)} of Lemma~\ref{lem:expandermixing} gives that $|P| \geq n - 2\lambda n/d = n -o(n)$.

Let us also denote $P$ by an ordering $v_1v_2 \ldots v_{l-1}v_l$, with $x := v_1,y := v_l$ as its endpoints. The first step is to apply Lemma~\ref{lem:findingcleanpartition}. Indeed, we partition $P$ into two subpaths $P_1,P_2$ of size $|P|/2$. By Lemma~\ref{lem:findingcleanpartition}, both paths have $\left(\delta,\gamma ,k \right)$-clean collections - indeed note that $\frac{10^{-8} d^2}{k \lambda^2} \geq \frac{C^2}{30 \cdot 10^{8} (\log n)^{1/3}} \geq \gamma$ if $C$ is large enough. Let $\mathcal{Q}^1 = \{Q^1_1, \ldots, Q^1_k\}, \mathcal{Q}^2 = \{Q^2_1, \ldots, Q^2_k\}$ be the given collection of disjoint sub-paths of $P_1, P_2$ respectively, and $S_1 \subseteq \bigcup_i Q^1_i, S_2 \subseteq \bigcup_i Q^2_i$ be the subsets with the given properties - note that $|\text{int}_P(S_1)|,|\text{int}_P(S_2)| \geq 0.4|P|$ and that $S_1,S_2$ are disjoint. 

First, we note that we can use Lemma \ref{lem:usualrotation1} in order to pass to a path $P'$ which has special properties, one of them being that one of its endpoints belongs to $S_1$ and the other to $S_2$. 
\begin{claim}\label{firstclaim}
There exists a path $P'$ on the same vertex set as $P$ with endpoints $x',y'$ such that the following hold for some $a \in \{1,2\}$ (let $b \in \{1,2\}$ with $b \neq a$).
\begin{itemize}
    \item $x' \in S_a, y' \in S_b$.
    \item $|\mathrm{dif}(P,P')| \leq 10 \log n$ and $|\mathrm{int}_P(S_b) \setminus \mathrm{int}_{P'}(S_b)| \leq 9$.
    \item $\mathrm{int}_P(S_a) \setminus \mathrm{int}_{P'}(S_a)$ contains at most $D/2$ vertices $z$ such that $\mathrm{dist}(z,x') \leq 2\log \log n$.
\end{itemize}
\end{claim}
\begin{proof}
Applying Corollary~\ref{cor:usualrotation} directly to $P$ and the set $R = S_1 \cup S_2$, we can find a $(\log n)$-rotation $P^1$ of $P$ with fixed endpoint $y$ such that the other endpoint, say $x'$, belongs to some $S_a$ with $a \in \{1,2\}$. Furthermore, we have that $|\text{int}_P(S_1 \cup S_2) \setminus \text{int}_{P^1}(S_1 \cup S_2)| \leq 3$. Similarly, we can again apply that corollary to $P^1$ and the set $R = S_1 \cup S_2$ to find a $(\log n)$-rotation $P^2$ of $P^1$ with fixed endpoint $x'$ such that the other endpoint, say $z$, belongs to $S_1 \cup S_2$. Furthermore, we have that $|\text{int}_P(S_1 \cup S_2) \setminus \text{int}_{P^2}(S_1 \cup S_2)| \leq 6$. Now, if $z \in S_b$, where $b \in \{1,2\}$ and $b \neq a$, then we are done by letting $y' := z$ and $P' := P^2$. If $z \in S_a$, we need still to apply Corollary \ref{cor:rotationinset} to $P^2$.

Let us firstly consider the case that $\delta \geq 100(\log n)^{4/3}$. Note that since $|\text{int}_P(S_1 \cup S_2) \setminus \text{int}_{P^2}(S_1 \cup S_2)| \leq 6$, we have that $S_a$ is $\left(P^2,\delta - 6 \right)$-clean and so, we can rotate the endpoint $z$ and use Corollary \ref{cor:rotationinset} to conclude that there are at least $n/10000$ vertices $w \in S_a$ for which there exists a $wx'$-path $P_w$ which is an $(S_a,\log n)$-rotation of $P^2$ with fixed endpoint $x'$. Since $|\text{int}_{P^2}(S_b)| \geq |P|/3 \geq n/4$, then property \textit{(4)} of Lemma \ref{lem:expandermixing} implies that there is an edge between such a $w$ and $\text{int}_{P_2}(S_b)$ which in turn, implies that there is some $y' \in S_b$ and a $y'x'$-path $P'$ which is a $(\log n + 1)$-rotation of $P^2$ with $x'$ as a fixed endpoint. Note that all the conditions are satisfied: $y' \in S_b$; because of the assumption on $\delta$ we have that $|\text{dif}(P,P')| \leq 10 \log n \leq D/2$ and so, the first part of the second condition and the third condition are satisfied; and since the path $P_w$ was an $S_a$-rotation of $P^2$ and $S_a,S_b$ are disjoint, we have that any change to the interior of $S_b$ only occurred with the last rotation step resulting in 
$P'$, hence $|\text{int}_P(S_b) \setminus \text{int}_{P'}(S_b)| \leq |\text{int}_P(S_b) \setminus \text{int}_{P^2}(S_b)| + 3 \leq 9$.

Now, let us consider when $\delta < 100(\log n)^{4/3}$ and so, $d < 10000(\log n)^{4/3} $ - recall that we always have $d \geq C^2(\log n)^{2/3}/2$ (since necessarily $d/(\log n)^{1/3} \geq C\lambda \geq C\sqrt{d/2}$ by the remark after Theorem \ref{thm:spectralgap}) and thus, $D \geq (\log n)^{1/3}$. Since we want to ensure the last property of the statement, consider the set $Y \subseteq S_a$ consisting of all the vertices in $S_a$ which are at a distance at most $2\log \log n$ from $x'$ - by the assumption on $d$, we have that $|Y| \leq d^{2 \log \log n} < 2^{10 (\log \log n)^2}$. We now apply Corollary \ref{cor:rotationinset} again, this time with the above set $Y$. Indeed, as previously, recall that $S_a$ is $\left(P^2,\delta -6 \right)$-clean and so, the corollary implies that there are at least $n/10000$ vertices $w \in S_a$ for which there exists a $wx'$-path $P_w$ which is an $(S_a,\log n)$-rotation of $P^2$ with fixed endpoint $x'$ \emph{and} such that $$|\text{dif}(P_w,P^2) \cap Y| \leq 3\log |Y| \leq 30 (\log \log n)^2 \leq D/2 - 6 .$$ 
Since $|\text{int}_{P_2}(S_b)| \geq n/3$, then property \textit{(4)} of Lemma \ref{lem:expandermixing} implies that there is an edge between such a $w$ and $\text{int}_{P_2}(S_b)$. To conclude, this implies that there is some $y' \in S_b$ and a $y'x'$-path $P'$ which is a $(\log n + 1)$-rotation of $P^2$ with $x'$ as a fixed endpoint and such that $\text{dif}(P',P^2) \cap Y$ has size at most $D/2 - 6$. Since we have that $|\text{int}_P(S_a) \setminus \text{int}_{P^2}(S_a)| \leq 6$ then it must be that $|\left(\text{int}_P(S_a) \setminus \text{int}_{P'}(S_a) \right)  \cap Y| \leq D/2$, as desired. The other conditions can be checked as in the first case.
\end{proof}

\noindent Let us define $F := \text{dif}(P,P')$ so that by the second property of the claim above, we have $|F| \leq k/3$. Now, consider the given sub-paths $Q^{a}_j, Q^{b}_j$ of $P$ and note that these might not be sub-paths of $P'$. Clearly however, $P'[Q^{a}_j], P'[Q^{b}_j]$ are disjoint unions of paths - we will denote these components of $P'[Q^{a}_j]$ by $C(Q^{a}_j,1), \ldots, C(Q^{a}_j,i_{a,j})$ and the components of $P'[Q^{b}_j]$ by $C(Q^{b}_j,1), \ldots, C(Q^{b}_j,i_{b,j})$. We will refer to them as the \emph{parts of $Q^{a}_j, Q^{b}_j$} (in $P'$), respectively. We will further say that $Q^{a}_j$ is a \emph{broken sub-path} if $i_{a,j} > 1$ and \emph{unbroken} otherwise (and analogously for $b$). Note that we must have $2k \leq \sum_{1 \leq j \leq k} i_{a,j} + \sum_{1 \leq j \leq k} i_{b,j}  \leq 2k + |F| + 1 \leq 7k/3$ + 1 and in particular, there are at most $|F| + 1$ parts of broken sub-paths. 

The goal will be to achieve the setting of Lemma \ref{lem:closingpathintocycle}. For this, we will first prove the following claim with a standard averaging argument. 
\begin{claim}\label{claim:averageinterval}
There exist two disjoint subpaths $A, B \subseteq P'$ such that both contain precisely $\gamma k/4$ parts of sub-paths in $\mathcal{Q}_1 \cup \mathcal{Q}_2$ and at least $\gamma k /8$ unbroken paths.
\end{claim}
\begin{proof}
Let us first take all the paths $C(Q^a_j,l), C(Q^b_j,l)$ and denote them by $C_1, \ldots, C_r$ where the order is given by the order in which they appear in the path $P'$ (for this, fix a direction of $P'$ as $x' \rightarrow P' \rightarrow y'$). Recall that $r = \sum_{1 \leq j \leq k} i_{a,j} + \sum_{1 \leq j \leq k} i_{b,j}  \in \left[2k, 7k/3 + 1 \right]$. Clearly, we can use this order to define a cover of $\bigcup_t C_t$ with disjoint sub-paths of $P'$ denoted by $P'_1,P'_2, \ldots$, such that each $P'_i$ contains $r_i$ consecutive sub-paths $C_j$ for some $r_i \in [\gamma k/4,\gamma k/2]$ - by construction, there are at least $\frac{r}{\max_i r_i} \geq \frac{2r}{\gamma k}$ paths $P'_i$.

Now, let us count how many paths $P'_j$ do not satisfy the property in the statement. Note first that every sub-path $P'_i$ has at least $\gamma k/4$ parts of sub-paths in $\mathcal{Q}_1 \cup \mathcal{Q}_2$ by construction. Then, if $P'_j$ does not contain a sub-path $P''_j \subseteq P'_j$ satisfying the property, it must contain at least $\gamma k/8$ parts from broken sub-paths. In turn, recall that we observed that the number of such parts must be at most $|F| + 1 \leq k/3 + 1$ and so, since the sub-paths $P'_i$ cover all parts of sub-paths in $\mathcal{Q}_1 \cup \mathcal{Q}_2$, we have that at most $\frac{|F|+1}{\gamma k/8} \leq 3/ \gamma$ sub-paths $P'_i$ do not satisfy the property. Concluding, since $ 3/\gamma < \frac{4}{\gamma} - 2 \leq \frac{2r}{\gamma k} - 2$ and there are at least $\frac{2r}{\gamma k}$ paths $P'_i$, we have the existence of the two desired subpaths $A,B$.
\end{proof}
We are now almost in the setting where we can apply Lemmas \ref{lem:technicalrotation} and \ref{lem:closingpathintocycle} in order to finish the proof. Indeed, note first the following.
\begin{claim}
$S_a \setminus (A \cup B), S_b \setminus (A \cup B) \subseteq P$ are $\left(P, D \right)$-clean subsets.
\end{claim}
\begin{proof}
We only prove this for $S_a$, since it is analogous for $S_b$. Let then $v \in S_a$ and recall that the third condition of Definition \ref{def:cleanedcollection} implies that there exist at least $\gamma k$ many sub-paths $Q^a_j$ such that $|N(v) \cap \text{int}_P(S_a \cap Q^a_j)| \geq \delta/k$. Recall also that Claim \ref{claim:averageinterval} implies that each $A$ and $B$ contain precisely $\gamma k/4$ parts from sub-paths in $\mathcal{Q}_a$ and so, there are at least $\gamma k/2$ sub-paths $Q^a_j$ disjoint to $A \cup B$ so that $|N(v) \cap \text{int}_P(S_a \cap Q^a_j)| \geq \delta/k$. Clearly, this implies that
$$|N(v) \cap \text{int}_P(S_a \setminus (A \cup B))| \geq (\gamma k/2) \cdot (\delta/k) \geq D ,$$
thus showing that $S_a \setminus (A \cup B)$ is $(P,D)$-clean.
\end{proof}
Note now that by Claim~\ref{claim:averageinterval}, the given sub-paths $A, B$ of $P'$ contain both at least $\gamma k/8$ unbroken paths $Q^{a}_i$, each of which have size at least $0.99|P|/k$ and so, we must have that $|A|,|B| \geq (\gamma k/8) \cdot (0.99|P|/k) > \gamma |P|/10 \geq 10 \lambda n/d$ (since $|P| \geq n - o(n)$). Therefore, Lemma \ref{lem:cleaningrotation} implies that there exist $P'$-clean subsets $A' \subseteq A, B' \subseteq B$ of size at least $0.9|A|$ and $0.9|B|$ respectively. Finally, this allows us to get to the following setting in which we can apply Lemma \ref{lem:closingpathintocycle}.
\begin{claim}\label{claim:settingoflemma}
There is a path $P''$ on the same vertex set as $P'$ with $|\mathrm{dif}(P',P'')| \leq \log n$ and the following properties.
\begin{enumerate}
    \item[(1)] The endpoints of $P''$ belong to $A'$ and $B'$.
    \item [(2)] Both $A',B' \subseteq P''$ are $(P'',800\lambda)$-clean.
    \item[(3)] $c_{P''}(A), c_{P''}(B) \leq 2$.
\end{enumerate}
\end{claim}
\begin{proof}
For this, we apply Lemma \ref{lem:technicalrotation}. Indeed, let $X_a := S_a \setminus (A \cup B)$, $X_b := S_b \setminus (A \cup B)$ which are disjoint and by the previous claim are $( P,D)$-clean. From Claim \ref{firstclaim} we also know that $|\text{int}_P(S_b) \setminus \text{int}_{P'}(S_b)| \leq 9$, which implies then that $X_b$ is $(P',D-9)$-clean, and that $\text{int}_P(X_a) \setminus \text{int}_{P'}(X_a)$ contains at most $D/2$ vertices $z$ such that $\text{dist}(z,x') \leq 2\log \log n $ - and note clearly (for a future application of Lemma \ref{lem:technicalrotation}) that $ 2\log |F| \leq 2\log \log n $. 

Therefore, Lemma \ref{lem:technicalrotation} implies firstly that there exist at least $\lambda n/d$ vertices $z \in X_a$ for which there exists a $zy'$-path which is an $(X_a, \log n)$-rotation of $P'$ with fixed endpoint $y'$ (note that this lemma can be applied since $D = d/(\log n)^{1/3} \geq C \lambda$ and $C$ is large enough). Let $Z_a$ denote the set of such vertices $z$, so that $|Z_a| \geq \lambda n/d$. Since $A$ is a sub-path of $P'$, then $|\text{int}_{P'}(A')| \geq |A| - 3|A \setminus A'| \geq 0.7|A| \geq \lambda n/d$ and therefore, there exists an edge between $Z_a$ and $\text{int}_{P'}(A')$. Crucially, by definition of $Z_a$ this implies that there exists a vertex $x'' \in A'$ such that there is an $(X_a \cup A, \log n+1)$-rotation of $P'$ with fixed endpoint $y'$ which is an $x''y'$-path, which we shall call $P'_a$. Moreover, it must be that $c_{P'_a}(A) \leq 2$, since the rotation only touches $A$ at the last step. Finally, note also that $\text{dif}(P',P'_a) \cap (X_b \cup B) = \emptyset$ and in particular, $B$ is still the same sub-path in $P'_a$.

Secondly, Corollary  \ref{cor:rotationinset} applied with $X_b$ implies that there exist at least $\lambda n/d$ vertices $z \in X_b$ for which there exists a $x''z$-path which is an $(X_b, \log n)$-rotation of $P'_a$ with fixed endpoint $x''$. Let $Z_b$ denote the set of such vertices $z$, so that $|Z_b| \geq \lambda n/d$. Much like before, since $|\text{int}_{P'_a}(B')| = |\text{int}_{P'}(B')| \geq |B| - 3|B \setminus B'| \geq 0.7|B| \geq \lambda n/d$, there exists an edge between $Z_b$ and $\text{int}_{P'_a}(B')$. Then, there is a vertex $y'' \in B'$ such that there is an $(X_b \cup B, \log n +1)$-rotation of $P'_a$ with fixed endpoint $x''$ which is a $x''y''$-path - this will be our desired path $P''$. Note also that like before, we have $c_{P''}(B) \leq 2$ and that $\text{dif}(P'_a,P'') \cap (X_a \cup A) = \emptyset$, so that $A$ is still the same in $P''$ as it was in $P'_a$, in particular, $c_{P''}(A) \leq 2$ as desired. Note that clearly properties \textit{(1),(3)} are satisfied. For property \textit{(2)} observe that since $A',B'$ were $P'$-clean and $c_{P''}(A), c_{P''}(B) \leq 2$, it must be that every vertex $v \in A'$ (\textit{respectively}, $v \in B'$) has at least $d|A'|/4n - 2 \geq \gamma d/100 \geq C^2 \lambda/100 \geq 800 \lambda$ (\textit{respectively}, $d|A'|/4n - 2 \geq 800 \lambda$) neighbours in $\text{int}_{P''}(A')$ (\textit{respectively}, $\text{int}_{P''}(B')$), and so, they are $(P'',800 \lambda)$-clean. Indeed, we used that $|A'| \geq 0.9|A|$, $|A| \geq \gamma |P|/10$, $|P| = n-o(n)$ and $\gamma \geq C^2 \lambda/d$.
\end{proof}
\noindent We can now finish immediately by applying Lemma~\ref{lem:closingpathintocycle} to $P''$.
\qed

\section{Linear forests with few paths and good endpoints}\label{sec:linforest}

In this section we will consider the problem of finding (in pseudorandom graphs) spanning linear forests with two special properties - they have a small number of paths; and all these paths have \emph{good} endpoints. As indicated in Section \ref{sec:outline}, this will be crucial in the proof of Theorem~\ref{thm:maindense}.

\begin{prop}\label{prop:forest}
Let $G$ be an $(n,d,\lambda)$-graph, $X \subseteq Y \subseteq V(G)$ and $\delta \geq 10^4 \lambda$ be such that $|X| \leq \delta n/1000d$ and both $Y,Y \setminus X$ are $\delta$-clean. Then, there exists a spanning linear forest $\mathcal{F}$ of $G[Y]$ such that the following hold.
\begin{itemize}
    \item The number of paths in $\mathcal{F}$ is $O \left( n/\delta^{1/5} \right)$.
    \item Every path in $\mathcal{F}$ has its endpoints in $Y \setminus X$.
\end{itemize}
\end{prop}
\noindent First, let us discuss the above statement in the case that $X = \emptyset$ - here we need only to find a spanning linear forest $\mathcal{F}$ of $G[Y]$ which does not have many components. 
This can be achieved as follows: Using the expansion properties of $G$, one can find a relatively dense regular subgraph of $G[Y]$. Invoking known results concerning the linear arboricity conjecture (see Theorem \ref{thm:arboricity} below), there is a decomposition of this subgraph into an almost optimal number of linear forests, and simply by averaging, one of these must consist of few paths.

In general, if $X$ is non-empty, we will first have to \emph{absorb} the vertices of $X$ into $Y \setminus X$, that is, we cover $X$ with a linear forest whose paths have their endpoints outside $X$. The idea is then to contract the pairs of endpoints and find a linear forest with few paths in the new graph. By reversing the contractions, we want to insert the paths covering $X$ into the paths of this linear forest. To make this feasible, we will have to work in a directed setting. Let us then first recall the following classical result of Alon~\cite{alon1988linear} (see also~\cite{alon2016probabilistic}) on the linear arboricity of regular digraphs, which will be crucial for us to find a spanning linear forest with few paths. By a \emph{$r$-regular digraph} we mean a digraph in which every vertex has in-degree and out-degree $r$.
\begin{thm}\label{thm:arboricity}
Every $r$-regular digraph can be covered by $r+O(r^{4/5})$ linear forests.
\end{thm}

\noindent We will now apply this theorem to find a linear forest with few components in digraphs with suitable expansion properties.
(We will later see that a random orientation of $G[Y]$ satisfies these properties.)

\begin{lem}\label{lem:linearforest}
Let $H$ be an $n$-vertex digraph and $r$ be such that the following holds. For every subdigraph $H' \subseteq H$ with $\Delta^{+}(H'), \Delta^{-}(H') < r$, we have that for all subsets $S \subseteq V(H)$,
$$|\Gamma^{+}_{H \setminus H'} \left( S \right)| \geq |S| .$$
Then, $H$ contains a spanning linear forest $\mathcal{F}$ with $O \left(n /r^{1/5} \right)$ paths.
\end{lem}
\begin{proof}
In order to construct $\mathcal{F}$, we will only need to show that $H$ contains an $r$-regular spanning subdigraph. Indeed, by Theorem~\ref{thm:arboricity}, such a digraph can be covered by $r+O(r^{4/5})$ linear forests. One of the linear forests in such a covering must contain at least $rn/(r+O(r^{4/5}))$ edges and hence consist of at most $O \left(n /r^{1/5} \right)$ paths, and therefore, $H$ will contain such a linear forest. 

Note we can construct the $r$-regular spanning subdigraph by finding $r$ edge-disjoint $1$-regular subdigraphs and therefore it is sufficient to show that for any $H' \subseteq H$ with $\Delta^{+}(H'), \Delta^{-}(H') < r$, the digraph $H \setminus H'$ contains a $1$-regular spanning subdigraph. For this, we define an auxiliary bipartite graph $B = (V_1,V_2)$ with both parts $V_1,V_2$ being copies of $V(H)$; a pair $xy$ with $x \in V_1, y\in V_2$ is an edge in $B$ if $x \rightarrow y$ is an edge of $H \setminus H'$. Now note that $B$ contains a perfect matching, since for every $S \subseteq V_1$, we have that $|N_B(S)| = |\Gamma^{+}_{H \setminus H'} \left( S \right)| \geq |S|$ by assumption, and therefore Theorem \ref{thm:hall} applies. To conclude, note that a perfect matching in $B$ corresponds precisely to a $1$-regular spanning subdigraph of $H \setminus H'$.
\end{proof}

\noindent The next lemma deals with the \emph{absorption} of the vertices of $X$. Moreover, its proof contains a novel covering idea which we believe may have other applications.
\begin{lem}\label{lem:coverforest}
There exists a linear forest $\mathcal{F'}$ in $G[Y]$ which covers $X$ and such that every path has its endpoints in $Y \setminus X$ and all its inner vertices in~$X$.
\end{lem}
\begin{proof}
The main idea is, instead of finding $\mathcal{F'}$ directly, we observe that it is sufficient to show the existence of a forest $\mathcal{F}$ in $G[Y]$ consisting of disjoint binary trees which all together cover $X$ and such that all the leaves are contained in $Y \setminus X$ and all the non-leaves in~$X$. Indeed, we can from $\mathcal{F}$ find the desired linear forest $\mathcal{F}'$ by doing the following operation. Consider some binary tree $T$ in $\mathcal{F}$ and take a path $P$ in $T$ which contains the root of $T$ and whose endpoints are leaves (observe that $P$ exists since the root of $T$ has degree two in $T$). Note that $\mathcal{F} - V(P)$ is still a disjoint union of binary trees and it covers $X \setminus V(P)$. Thus, we can repeat the same operation until all of $X$ has been covered by these paths, which altogether form the desired linear forest $\mathcal{F}'$ (we can in every step discard the vertices of the forest which become isolated, since those must belong to $Y\setminus X$).   

Now, in order to find the forest $\mathcal{F}$ let us note that there exists an ordering $x_1,x_2, \ldots, x_t$ of the vertices of $X$ such that for all $i$, we have
$$\left|N_G(x_i) \cap \left(\{x_1, \ldots, x_{i-1}\} \cup \left(Y \setminus X \right)\right)  \right| \geq \delta/2 .$$
\noindent Indeed, suppose that we have such a partial sequence $x_1, \ldots, x_{i-1}$. If no candidate for the vertex $x_i$ exists, then since $G[Y]$ has minimum degree at least $\delta$, we have that $G[X \setminus \{x_1, \ldots, x_{i-1}\}]$ has minimum degree at least $\delta - \delta/2 = \delta/2$. However, this contradicts part $\textit{(2)}$ of Lemma \ref{lem:expandermixing} since $|X \setminus \{x_1, \ldots, x_{i-1}\}| \leq |X| \leq \delta n/(1000d)$ and $\lambda \leq 10^{-4}\delta$.

Given this ordering we can do the following. Define an auxiliary bipartite graph $H$ with parts $A$ and $B$, where $A$ consists of two copies of each vertex in $X$ and $B = Y$. The edges are defined in the following manner - each copy of $x_i \in X$ in $A$ has as neighbours those vertices in $\{x_1, \ldots, x_{i-1}\} \cup \left(Y \setminus X \right)$ which are neighbours of $x_i$ in $G$. Crucially, it is easy to see that $H$ has the following property.
\begin{obs*}\label{obs:binarytree}
If there is a matching in $H$ covering $A$, then there is a forest $\mathcal{F}$ in $G[Y]$ which covers $X$ and whose components are binary trees with leaves in $Y \setminus X$ and non-leaves in $X$.
\end{obs*}
\noindent Since our objective is to find such an $\mathcal{F}$, we need only to find a matching in $H$ covering $A$ - for this we will check that the conditions of Theorem \ref{thm:hall} are satisfied. Consider any $S \subseteq A$ and let $S'$ be the corresponding set in $V(G)$ (not counting possible repetitions of a vertex in $S$) - note that $|S'| \geq |S|/2$.
By the properties of the ordering $x_1,x_2,\ldots, x_t$ of $X$ we have that $e_G(N_H(S),S') \geq |S'|\delta/2 \ge |S|\delta/4 $. Also, by part \textit{(1)} of Lemma \ref{lem:expandermixing} we have $e_G(N_H(S),S') \leq \frac{d|N_H(S)||S'|}{n} + \lambda \sqrt{|S'||N_{H}(S)|} \leq \delta |N_H(S)|/500 + \lambda \sqrt{|S||N_{H}(S)|} $ since $|S'| \leq |S|\le 2|X| \leq \delta n/500d$. Putting these together and using that $\delta \geq 100 \lambda$ gives that $|N_H(S)| \geq |S|$ as desired.
\end{proof}
\noindent Now that $X$ has been absorbed into $Y \setminus X$, we can define the auxiliary random directed graph mentioned before. We define a directed graph $H$ as follows. Denote the paths forming the linear forest $\mathcal{F}'$ given by Lemma \ref{lem:coverforest} by $P_1, \ldots, P_l$ and the endpoints of each path $P_i$ by $x_i,y_i \in Y \setminus X$ - let $U \subseteq Y \setminus X$ denote the set of all these endpoints. For each $i$, contract the pair $\{x_i,y_i\}$ into one vertex $z_i$ and choose uniformly at random between defining $z^{+}_i := x_i, z^{-}_i := y_i$ or $z^{+}_i := y_i, z^{-}_i := x_i$. Let $Z$ denote the set of vertices $z_i$ - the vertex set of $H$ is then $Z \cup Y \setminus (U \cup X)$. Note also that $|Z| \leq |X|$. The edges of $H$ are defined in the following manner.
\begin{enumerate}
    \item Every edge of $G[Y \setminus (U \cup X)]$ is oriented independently and uniformly at random.
    \item For two vertices $z_i,z_j$, we include the directed edge $z_i \rightarrow z_j$ if $z^{+}_i z^{-}_j$ is an edge of $G$. 
    \item For a vertex $z_i$ and a vertex $v \in Y \setminus (U \cup X)$ we include the edge $z_i \rightarrow v$ (conversely, $v \rightarrow z_i$) if $z^{+}_i v$ (conversely, $ z^{-}_i v$) is an edge of $G$.
\end{enumerate}
\noindent Crucially, by replacing each vertex $z_i$ with the path $P_i$, it is easy to observe that $H$ has the following deterministic property.
\begin{obs}\label{obs:dirforest}
If there is a spanning linear forest in $H$, then there is a spanning linear forest $\mathcal{F}$ of $G[Y]$ with the same number of paths, all of which have their endpoints in $Y \setminus X$.
\end{obs}
\noindent Therefore, we will need only to show that with positive probability, $H$ satisfies the conditions of Lemma~\ref{lem:linearforest} with $r := \Theta(\delta) $. For convenience, we make first some parallel notation between $H$ and $G$. For any set $S$ consisting of vertices in $V(H)$, we let $S_{G}$ denote the set of vertices in $Y \subseteq V(G)$ given by substituting every $z_i \in S$ by the vertices $x_i,y_i$. Note that trivially $|S| \leq |S_{G}| \leq 2|S|$ and that for every sets $A, B \subseteq V(H)$ we have $\overrightarrow{e}_H(A,B) \leq e_{G}(A_G,B_G)$.

We now want to show that $H$ satisfies the conditions of Lemma \ref{lem:linearforest}. First, we note the following.
\begin{lem}\label{lem:Hexpansionlineararb}
With positive probability, the following all hold.
\begin{enumerate}
    \item $H$ has minimum out-degree and in-degree at least $\delta/10 $.
    \item For all $A,B \subseteq V(H)$ with $|A|,|B| \geq \delta n/100d$, there are at least $\delta|A|/4000$ edges directed from $A$ to $B$.
\end{enumerate}
\end{lem}
\begin{proof}
 We first deal with the degree conditions. Note that since by assumption, the set $Y \setminus X$ is $\delta$-clean in $G$, it is easy to check that the definition of $H$ implies that all vertices $v$ of $H$ are such that their out-degree and in-degree are random variables which stochastically dominate the binomial random variable $\text{Bin}(\delta,1/2)$. Indeed, note that if $v \notin U$ has both elements of some pair $\{x_i,y_i\}$ as neighbours in $G$, then $z_i$ will be both an in and out-neighbour of it in $H$; if only one of $x_i,y_i$ is a neighbour, then with probability $1/2$, $z_i$ will be an in/out-neighbour in $H$; if some other $v' \notin U$ is a neighbour, then also with probability $1/2$, it will be an in/out-neighbour in $H$. If $v = z_i$ for some $i$, then the analysis is similar. Therefore, the event $E_v$ that $v$ has out-degree or in-degree smaller than $\delta/10$ has $\mathbb{P}(E_v) = e^{-\Omega(\delta)}$. 
 Further, since $G$ is $d$-regular, the event $E_v$ depends on at most $O(d^2)$ other events $E_{v'}$. Indeed it is easy to see that the events $E_{v}$ is independent of all events $E_{v'}$ when $v'$ is not a neighbour of $v$ or there is no pair $\{x_i, y_i\}$  with both $v, v'$ having neighbours in $\{x_i, y_i\}$. On the other hand when $v, v'$ both have neighbours in the pair $\{x_i, y_i\}$, then, knowing the event $E_{v'}$
 affects the choice of $z^{+}_i,z^{-}_i$ and therefore also the out/in-degree of $v$.
 Because $e^{-\Omega(\delta)} = o(d^{-2})$, Lemma~\ref{lem:locallemma} now implies that with probability at least $\left(1-\frac{1}{O(d^2)} \right)^n > e^{-n}$, no event $E_v$ occurs and thus, the first condition holds. 

For the second condition, we apply a union bound. Indeed, let $A,B \subseteq V(H)$ be sets of size at least $\delta n/100d$ - note that we have at most $4^n$ such pairs. Further, consider the sets $A' := A \setminus Z$ and $B' := B \setminus Z$, which, since $|Z| \leq |X| \leq \delta n/1000d$, have size at least $\delta n/200d$. The random procedure defining $H$ implies that $\overrightarrow e_H(A',B')$ has the distribution of the binomial random variable $\text{Bin} \left(e_G(A',B') ,1/2\right)$. By Lemma~\ref{lem:expandermixing} we have that
$$e_G(A',B') \geq d|A'||B'|/n - \lambda \sqrt{|A'||B'|} > d|A'||B'|/2n \geq \delta |A'|/400 > \delta|A|/1000 > 500 n ,$$
where the second inequality follows since $\delta \geq 10000 \lambda$ and the last inequality since $\delta \geq 10000 \lambda \geq 10000 \sqrt{d/2}$ (recalling the remark after Theorem \ref{thm:spectralgap}). So, by Lemma~\ref{lem:chernoff}, with probability at most $e^{-3n}$ we have that $\overrightarrow e_H(A',B') < \delta|A|/4000$. By a union bound, we then have that with probability at least $1-4^n \cdot e^{-3n} > 1 - e^{-n}$ all such pairs $A,B$ have $\overrightarrow e_H(A,B) \geq \delta|A|/4000$.

Now, we can combine both conditions. Indeed, we showed that the first occurs with probability larger than $e^{-n}$ and the second with probability larger than $1 - e^{-n}$. Thus, with positive probability both occur.
\end{proof}
We can now prove Proposition~\ref{prop:forest}.
\begin{proof}[ of Proposition \ref{prop:forest}]
Assume that $H$ satisfies the conditions \textit{1} and \textit{2} of the lemma above.
Let us then verify the conditions of Lemma~\ref{lem:linearforest} with $r := \delta/10^{10}$. Take $H' \subseteq H$ to be a subdigraph such that $\Delta^{+}(H'),\Delta^{-}(H') < r$ and let $S \subseteq V(H)$, $T := \Gamma^{+}_{H \setminus H'}(S)$. Suppose for sake of contradiction that $|T| < |S|$. 

First, note that by condition \textit{1} of the lemma above, we have that $\overrightarrow{e}_H\left(S,T \right) \geq \delta|S|/20$ since the out-degree in $H \setminus H'$ is still at least $\delta/10 - r \geq \delta/20$.
In turn Lemma~\ref{lem:expandermixing} implies that
$$\overrightarrow{e}_H \left(S,T \right) \leq e_{G} \left(S_{G},T_{G} \right) \leq \frac{d|S_{G}||T_{G} |}{n} + \lambda \sqrt{|S_{G}||T_{G}|} < |S| \left(\frac{4d|S|}{n} + 2\lambda \right) ,$$ where we are using that $|S_{G}| \leq 2|S|$, $|T_{G}| \leq 2|T|$ and $|S| > |T|$. Observe now that the two inequalities above give a contradiction when $|S| \leq \delta n/100d$ since $\delta > 10000 \lambda$. Hence, we can assume that $|S| > \delta n/100d$. Furthermore, note that the above argument also shows that if $|S| \leq \delta n/100d$, then $\Gamma^{-}_{H \setminus H'}(S)$ has size at least $|S|$, since condition \textit{1} of the lemma above also gives that every vertex has in-degree at least $\delta/10$.

Now, suppose that $|S| > \delta n/100d$ and note that part \textit{2} of Lemma \ref{lem:Hexpansionlineararb} implies that $|T| > |V(H)| - \delta n/100d$. Indeed, we have for all sets $B$ of size at least $\delta n/100d$ that $\overrightarrow{e}_H(S,B) > r|S|$ and so, $\overrightarrow{e}_{H \setminus H'}(S,B) \neq \emptyset$. Therefore, we can further assume that $|S| > |V(H)| - \delta n/100d$, since otherwise $|T| \geq |S|$. In this case since $|T| < |S|$, we can consider a set $T' \subseteq V(H) \setminus T$ of size
$|V(H)|-|S|<|T'| \leq \delta n/100d$. By the previous paragraph, we then have that $\Gamma^{-}_{H \setminus H'}(T')$ has size at least $|T'|$. Also note that by definition of $T'$, we have $\Gamma^{-}_{H \setminus H'}(T') \cap S = \emptyset$. This implies that $|T'| + |S| \leq |V(H)|$ which is a contradiction since we have $|S| > |V(H)| - |T'|$ by assumption. Concluding, we must always have $|T| \geq |S|$ and so, Lemma~\ref{lem:linearforest} implies that $H$ contains the desired spanning linear forest with positive probability. By Observation~\ref{obs:dirforest}, this implies the existence of the desired linear forest $\mathcal{F}$ in $G[Y]$.
\end{proof} 

\section{Hamilton cycles from good collections of paths}\label{sec:manycyclesthm}
In this section, we will prove Theorem~\ref{thm:maindense}. More generally, we will prove that pseudorandom graphs which contain a certain spanning structure must also contain a Hamilton cycle. We will now describe this structure. First, given a path $P$ and a collection of vertex-disjoint cycles $\mathcal{C} = \{C_1, \ldots, C_l\}$, we say that $P$ \emph{connects} $\mathcal{C}$ if there exists a choice of edges $e_i \in C_i$ for every $i$ so that $P$ contains every edge $e_i$ and is disjoint to the rest. 
Given such a setting, for each $i$ we let $x^{P}_i,y^{P}_i \in C_i \setminus e_i$ denote the vertices of $C_i$ which are not incident to $e_i$ but are closest to this edge in the cycle - with the possibility that $x^{P}_i = y^{P}_i$ in the case that $C_i$ only has three vertices. The set of these vertices is denoted by $V_{P,\mathcal{C}} := \{x^{P}_i,y^{P}_i: 1 \leq i \leq l\}$ and we shall refer to it as the \emph{flexible set}.
As outlined in Section~\ref{sec:outline}, the path $P$ is able to \emph{absorb}, for every $i$ independently, the vertices of $V(C_i\setminus e)$ into the path, by replacing $e_i$ with the path $C_i-e$. 
This flexibility allows us to connect partial structures outside of $P$. Since each connection requires logarithmic length in general, we are only able to do few connections using the flexible set, so we will assume in addition that the remaining vertices are already partitioned into few paths. This motivates the following definition.

\begin{defn}\label{def:perfectlinearforest}
A collection $\mathcal{F}$ of vertex-disjoint paths in a graph $G$ is said to be \emph{$(r,l,\delta)$-good} if it consists of $r$ paths $P_1, P_2, \ldots, P_r$ with the following properties.
\begin{enumerate}
    \item The path $P_1$ connects a collection $\mathcal{C}$ of $l$ vertex-disjoint cycles, which are disjoint to the paths $P_2, \ldots, P_r$. Furthermore, these paths together with $\mathcal{C}$ span the whole vertex set.
    \item $V_{P_1, \mathcal{C}}$ induces a graph with minimum degree at least $\delta$.
    \item The endpoints of the paths $P_1, \ldots, P_r$ have each at least $\delta$ neighbours in the flexible set $V_{P_1, \mathcal{C}}$.
\end{enumerate}
\end{defn}

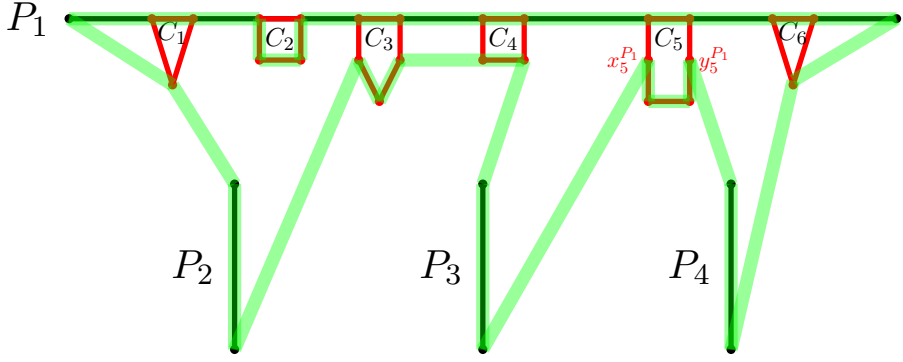
\begin{figure}[h]
    \centering
    \begin{tikzpicture}[scale=1.1,main node/.style={circle,draw,color=black,fill=black,inner sep=0pt,minimum width=3pt}]
        \tikzset{cross/.style={cross out, draw=black, fill=none, minimum size=2*(#1-\pgflinewidth), inner sep=0pt, outer sep=0pt}, cross/.default={2pt}}
	\tikzset{rectangle/.append style={draw=brown, ultra thick, fill=red!30}}
	    
\draw[line width= 2 pt] (-5,2) to (5,2);
\draw[line width= 2 pt] (-3,-2) to (-3,0);
\draw[line width= 2 pt] (0,-2) to (0,0);
\draw[line width= 2 pt] (3,-2) to (3,0);

\node[color=black, scale=3 ]  at (-5,2) {$.$};
\node[color=black, scale=3 ]  at (5,2) {$.$};
\node[color=black, scale=1.4 ]  at (-5.5,2) {$P_1$};
\node[color=black, scale=3 ]  at (-3,0) {$.$};
\node[color=black, scale=3 ]  at (-3,-2) {$.$};
\node[color=black, scale=1.4 ]  at (-3.5,-1){$P_2$};
\node[color=black, scale=3 ]  at (0,0) {$.$};
\node[color=black, scale=3 ]  at (0,-2) {$.$};
\node[color=black, scale=1.4 ]  at (-0.5,-1) {$P_3$};
\node[color=black, scale=3 ]  at (3,0) {$.$};
\node[color=black, scale=3 ]  at (3,-2) {$.$};
\node[color=black, scale=1.4 ]  at(2.5,-1){$P_4$};

\node[color=red, scale=3 ] at(-4,2){$.$};
\node[color=red, scale=3 ]  at(-3.5,2){$.$};
\node[color=red, scale=3 ]  at(-3.75,1.2){$.$};
\draw[color = red, line width= 2 pt] (-4,2) to (-3.5,2);
\draw[color = red, line width= 2 pt] (-3.75,1.2) to (-3.5,2);
\draw[color = red, line width= 2 pt] (-4,2) to (-3.75,1.2);
\node[color=black, scale=0.9 ]  at(-3.75,1.83){$C_1$};

\node[color=red, scale=3 ] at(-2.7,2){$.$};
\node[color=red, scale=3 ]  at(-2.2,2){$.$};
\node[color=red, scale=3 ]  at(-2.7,1.5){$.$};
\node[color=red, scale=3 ]  at(-2.2,1.5){$.$};
\node[color=black, scale=0.9 ]  at(-2.45,1.75){$C_2$};
\draw[color = red, line width= 2 pt] (-2.2,2) to (-2.7,2);
\draw[color = red, line width= 2 pt] (-2.2,2) to (-2.2,1.5);
\draw[color = red, line width= 2 pt] (-2.7,2) to (-2.7,1.5);
\draw[color = red, line width= 2 pt] (-2.2,1.5) to (-2.7,1.5);

\node[color=red, scale=3 ] at(-1.5,2){$.$};
\node[color=red, scale=3 ]  at(-1,2){$.$};
\node[color=red, scale=3 ]  at(-1.5,1.5){$.$};
\node[color=red, scale=3 ]  at(-1,1.5){$.$};
\node[color=red, scale=3 ]  at(-1.25,1){$.$};
\draw[color = red, line width= 2 pt] (-1.5,2) to (-1,2);
\draw[color = red, line width= 2 pt] (-1.5,2) to (-1.5,1.5);
\draw[color = red, line width= 2 pt] (-1.5,1.5) to (-1.25,1);
\draw[color = red, line width= 2 pt] (-1,1.5) to (-1.25,1);
\draw[color = red, line width= 2 pt] (-1,2) to (-1,1.5);
\node[color=black, scale=0.9 ]  at(-1.25,1.75){$C_3$};

\node[color=red, scale=3 ] at(0,2){$.$};
\node[color=red, scale=3 ]  at(0.5,2){$.$};
\node[color=red, scale=3 ]  at(0.5,1.5){$.$};
\node[color=red, scale=3 ]  at(0,1.5){$.$};
\draw[color = red, line width= 2 pt] (0.5,2) to (0,2);
\draw[color = red, line width= 2 pt] (0.5,2) to (0.5,1.5);
\draw[color = red, line width= 2 pt] (0,2) to (0,1.5);
\draw[color = red, line width= 2 pt] (0,1.5) to (0.5,1.5);
\node[color=black, scale=0.9 ]  at(0.25,1.75){$C_4$};

\node[color=red, scale=3 ] at(2,2){$.$};
\node[color=red, scale=3 ]  at(2.5,2){$.$};
\node[color=red, scale=3 ]  at(2,1.5){$.$};
\node[color=red, scale=3 ]  at(2.5,1.5){$.$};
\node[color=red, scale=3 ]  at(2,1){$.$};
\node[color=red, scale=3 ]  at(2.5,1){$.$};
\draw[color = red, line width= 2 pt] (2,2) to (2,1.5);
\draw[color = red, line width= 2 pt] (2,1) to (2,1.5);
\draw[color = red, line width= 2 pt] (2.5,1) to (2.5,1.5);
\draw[color = red, line width= 2 pt] (2.5,2) to (2.5,1.5);
\draw[color = red, line width= 2 pt] (2,2) to (2.5,2);
\draw[color = red, line width= 2 pt] (2,1) to (2.5,1);
\node[color=red, scale=0.75 ]  at(1.7,1.5){$x^{P_1}_5$};
\node[color=red, scale=0.75 ]  at(2.8,1.5){$y^{P_1}_5$};
\node[color=black, scale=0.9 ]  at(2.25,1.75){$C_5$};

\node[color=red, scale=3 ] at(3.5,2){$.$};
\node[color=red, scale=3 ]  at(4,2){$.$};
\node[color=red, scale=3 ]  at(3.75,1.2){$.$};
\draw[color = red, line width= 2 pt] (3.5,2) to (3.75,1.2);
\draw[color = red, line width= 2 pt] (4,2) to (3.75,1.2);
\draw[color = red, line width= 2 pt] (4,2) to (3.5,2);
\node[color=black, scale=0.9 ]  at(3.75,1.83){$C_6$};

\draw[color = green, opacity = 0.4, line width= 5 pt] (-5,2) to (-3.75,1.2);
\draw[color = green, opacity = 0.4, line width= 5 pt] (-3,0) to (-3.75,1.2);
\draw[color = green, opacity = 0.4, line width= 5 pt] (-3,0) to (-3,-2);
\draw[color = green, opacity = 0.4, line width= 5 pt] (-1.5,1.5) to (-3,-2);
\draw[color = green, opacity = 0.4, line width= 5 pt] (-1.5,1.5) to (-1.25,1);
\draw[color = green, opacity = 0.4, line width= 5 pt] (-1,1.5) to (-1.25,1);
\draw[color = green, opacity = 0.4, line width= 5 pt] (-1,1.5) to (0,1.5);
\draw[color = green, opacity = 0.4, line width= 5 pt] (0.5,1.5) to (0,1.5);
\draw[color = green, opacity = 0.4, line width= 5 pt] (0.5,1.5) to (0,0);
\draw[color = green, opacity = 0.4, line width= 5 pt] (0,-2) to (0,0);
\draw[color = green, opacity = 0.4, line width= 5 pt] (0,-2) to (2,1.5);
\draw[color = green, opacity = 0.4, line width= 5 pt] (2,1.5) to (2,1);
\draw[color = green, opacity = 0.4, line width= 5 pt] (2.5,1) to (2,1);
\draw[color = green, opacity = 0.4, line width= 5 pt] (2.5,1) to (2.5,1.5);
\draw[color = green, opacity = 0.4, line width= 5 pt] (3,0) to (2.5,1.5);
\draw[color = green, opacity = 0.4, line width= 5 pt] (3,0) to (3,-2);
\draw[color = green, opacity = 0.4, line width= 5 pt] (3.75,1.25) to (3,-2);
\draw[color = green, opacity = 0.4, line width= 5 pt] (3.75,1.25) to (5,2);
\draw[color = green, opacity = 0.4, line width= 5 pt] (-2.2,2) to (5,2);
\draw[color = green, opacity = 0.4, line width= 5 pt] (-2.2,2) to (-2.2,1.5);
\draw[color = green, opacity = 0.4, line width= 5 pt] (-2.7,1.5) to (-2.2,1.5);
\draw[color = green, opacity = 0.4, line width= 5 pt] (-2.7,2) to (-2.7,1.5);
\draw[color = green, opacity = 0.4, line width= 5 pt] (-5,2) to (-2.7,2);

    \end{tikzpicture}
    \caption{An illustration of a good collection of paths and how it can be used to form a Hamilton cycle. The path $P_1$ connects the cycles $C_1,C_2, \ldots, C_6$ and together with the paths $P_2,P_3,P_4$ it spans the whole vertex set. A Hamilton cycle is drawn in green by connecting the endpoints of the paths $P_1,P_2,P_3,P_4$ using the cycles. Those cycles which are not used for connecting (in this case $C_2$) are then absorbed into the path $P_1$.}
    \label{fig:perfectlinearforest}
\end{figure}

\noindent From the above definition and the lemmas presented in Section \ref{sec:altpaths}, we can show that pseudorandom graphs which contain a good collection of paths must also be Hamiltonian. The reader might want to refer to Figure~\ref{fig:perfectlinearforest} for an illustration.
\begin{thm}\label{thm:hamcycleplf}
Let $G$ be an $(n,d,\lambda)$-graph with $\lambda < d/500$. Let $\mathcal{F}$ be an $\left(\frac{l}{100\log n}, l,500 \lambda \right)$-good collection of paths on $V(G)$ for some $l \geq 2000 \lambda n/d$. Then $G' := G \cup \mathcal{F}$ contains a Hamilton cycle.
\end{thm}
\begin{proof}
Let $\mathcal{F}$ consist of the paths $P_1, \ldots, P_r$ and a cycle collection $\mathcal{C} = \{C_1, \ldots,C_l \}$. For each cycle $C_j$, let $(x_j,y_j)$ denote the vertices $x^{P_1}_j,y^{P_1}_j \in C_j \setminus e_j$ as earlier defined and $\mathcal{M}$ denote the collection of these pairs. Further, for each path $P_i$ let $u_i,v_i$ denote its endpoints and define the collection $\mathcal{P}$ of pairs $\{a_j,b_j\}$ with $a_j := v_j, b_j := u_{j+1}$ (indices modulo $r$). Notice that we can apply Theorem \ref{thm:vertexdisjointpathsvariation}. Indeed, $V(\mathcal{M}) = V_{P_1, \mathcal{C}}$, and so, the first and second conditions imply that $V(\mathcal{M})$ is $500 \lambda$-clean and $|\mathcal{M}| = l \geq 2000 \lambda n/d$. Further, we have that $|\mathcal{P}| \leq \frac{|\mathcal{M}|}{100 \log n}$. Theorem \ref{thm:vertexdisjointpathsvariation} then implies that there exist vertex-disjoint $\mathcal{M}$-alternating paths between the pairs $\{a_j,b_j\}$ in the graph obtained by adding every non-existing edge $x_iy_i$ to $G$. In turn, it is easy to see that this produces a Hamilton cycle in $G'$.
\end{proof}

\subsection{Graphs with many cycles have a good collection of paths}
\noindent As a corollary of Theorem \ref{thm:hamcycleplf}, we can show that pseudorandom graphs which contain many cycles are Hamiltonian. Clearly, this will imply Theorem \ref{thm:maindense}.
\begin{thm}\label{thm:pflhamiltonian} Let $G$ be an $(n,d,\lambda)$-graph with $d \geq (\log n)^{11}$ and $\lambda < d/10^{10}$. Suppose that it contains $20000 \lambda n/d$ many vertex-disjoint cycles. Then it contains an $\left(\frac{l}{100\log n}, l,500 \lambda \right)$-good collection of paths for some $l \geq 2000 \lambda n/d$. Therefore, it is Hamiltonian.
\end{thm}
\begin{proof}
Let $G$ be an $(n,d,\lambda)$-graph and let $\mathcal{C} = \{C_1, \ldots, C_t\}$ denote the given collection of $t$ vertex-disjoint cycles, where $t \geq 20000 \lambda n/d$. Notice that we can also assume that $t \geq n/ (2\log n)$ since such a cycle collection will always exist if say, $d \geq 100 \lambda $ (this is easy to see by greedily finding cycles of size at most $\log n$). We will now first prove a series of preparatory lemmas in order to make this collection more suitable to work with. The following ensures that every vertex has many neighbours disjoint to these.
\begin{lem}\label{lem:degreesoutsidecycles}
There exists a subcollection $\mathcal{C}' \subseteq \mathcal{C}$ of at least $3t/4$ cycles such that every vertex in $G$ has at least $d/10$ neighbours not in $V(\mathcal{C}')$.
\end{lem}
\begin{proof}
First, we refine $\mathcal{C}$ so that it only contains short cycles. Indeed, at most $t/100$ of the cycles have length larger than $100n/t$. Let us delete all of these from $\mathcal{C}$. Let now $\mathcal{C}' \subseteq \mathcal{C}$ be randomly chosen by letting every cycle be included in it independently and with probability $4/5$. Using Lemma \ref{lem:chernoff}, we then have that with high probability, $|\mathcal{C}'| \geq 3t/4$ as required. Further, every vertex has degree $d \geq (\log n)^{11}$ and since each cycle is now of length at most $100n/t = O(\log n)$ (since by before, we can assume that $t \geq n/ (2\log n)$) we can apply Lemma \ref{lem:chernoff} with $k := 100n/t$, to get that the probability that a vertex has less than $d/10$ neighbours not in $V(\mathcal{C}')$ is $o(n^{-1})$. By a union bound over all vertices, we then have also that with high probability, every vertex has at least $d/10$ neighbours not in $V(\mathcal{C}')$. Combining both of these properties gives the desired outcome for some choice of $\mathcal{C}'$.
\end{proof}
\noindent Given the above, let us redefine $\mathcal{C}$ to be $\mathcal{C'}$. We now need the following claim which deals with finding a path connecting a decent proportion of the cycles in $\mathcal{C}$.
\begin{lem}
There exists a subcollection $\mathcal{C}' \subseteq \mathcal{C}$ of $t/2$ cycles and a path $P$ connecting them. Furthermore, among every two consecutive edges of $P$ there is an edge belonging to some cycle in $\mathcal{C}'$.
\end{lem}
\begin{proof}
For each cycle $C_i \in \mathcal{C}$, pick an arbitrary edge $e_i = x_i y_i$ belonging to it. Let $\mathcal{M}$ denote the collection of pairs $\{x_i,y_i\}$. Now, by part \textit{(4)} of Lemma \ref{lem:expandermixing}, every two disjoint subsets of $G$ of size $\lambda n/d$ have an edge between them. Therefore, Lemma~\ref{lem:DFSaltpath} implies that $G \cup \mathcal{M}$ contains an $\mathcal{M}$-alternating path which uses all but at most $2\lambda n/d - 1$ of the edges $e_i$. Since $|\mathcal{C}| - 2\lambda n/d \geq 3t/4 - 2\lambda n/d \geq t/2$, this gives the desired path~$P$.

\end{proof}
\noindent Observe now that since among every two consecutive edges of $P$ there is an edge belonging to some cycle of $\mathcal{C}'$, we have from Lemma~\ref{lem:degreesoutsidecycles} that every vertex has at least $d/10$ neighbours not in $V(P) \cup V(\mathcal{C'})$. In order to continue let us again redefine $\mathcal{C}$ to be $\mathcal{C'}$. 

Now, let us denote $\mathcal{C}$ by $\{C_1, \ldots, C_r\}$ with $r \geq t/2$ and that the edges $e_1, \ldots,e_r$ appear in this order in the path $P$. We further delete the cycles $C_1, \ldots ,C_{r/3},C_{2r/3}, \ldots, C_r$ from $\mathcal{C}$ so that all the cycles now in consideration are present in the \emph{middle} of the path $P$. More precisely, no cycle in $\mathcal{C}$ will intersect the first $r/3 \geq 2\lambda n/d$ or last $r/3 \geq 2 \lambda n/d$ vertices of the path $P$ - let us mark this as property ($\ast$). Now, since $r/3 \geq t/6 \geq 3000\lambda n/d$, Lemma \ref{lem:paircleaning} directly implies the following.
\begin{lem}\label{lem:reservoir}
There exists a subcollection $\mathcal{C'} \subseteq \mathcal{C}$ of at least $t/7$ cycles such that $V_{P,\mathcal{C'}}$ is $500\lambda$-clean.
\end{lem}

\noindent Again, let us redefine $\mathcal{C}$ to be the new collection $\mathcal{C'}$, which is the final collection of cycles - it has size now $l:=|\mathcal{C}| \geq t/7 \geq 2000 \lambda n/d$ and $|V_{P,\mathcal{C}}| \geq l$ as well. Let us note that by Lemma~\ref{lem:cleaning}, at most $4\lambda^2 n^2/d^2|V_{P,\mathcal{C}}| \leq \lambda n/d$ vertices have degree less than $d|V_{P,\mathcal{C}}|/2n \geq dl/2n \geq 500\lambda$ in $V_{P,\mathcal{C}}$. Therefore, among the first $2 \lambda n/d$ vertices of $P$ there is a vertex $x$ with degree at least $500\lambda$ in $V_{P,\mathcal{C}}$ and among the last vertices of $P$ there is also such a vertex $y$. By property ($\ast$), we can then make $P$ an $xy$-path and still have that $P$ connects $\mathcal{C}$, so that its endpoints have degree at least $500\lambda$ in $V_{P,\mathcal{C}}$. 

Finally, having constructed $P_1 := P$ and the collection $\mathcal{C}$, we are left to find the paths $P_2, P_3, \ldots, P_r$ and ensure the third condition of Definition \ref{def:perfectlinearforest}. For this we will apply Proposition \ref{prop:forest}. Note first that $Y := V(G) \setminus (V(P) \cup V(\mathcal{C}))$ is, by before, such that $G[Y]$ has minimum degree at least $d/10$ (which by Lemma~\ref{lem:expandermixing} clearly implies that $|Y| \geq n/11$). Moreover, let $X' := \{v \in Y: d(v,V_{P,\mathcal{C}}) < 500\lambda\}$. As previously noted, we have that $|X'| \leq \lambda n/d$. Furthermore, the second part of Lemma~\ref{lem:cleaning} applied to $S := Y \setminus X'$ implies that there is $X' \subseteq X \subseteq Y$ with $|X| \leq |X'| + \lambda n/d \leq 2 \lambda n/d$ such that $Y \setminus X$ is clean. Since as noted before, $|Y| \geq n/11$, then $G[Y \setminus X]$ has minimum degree at least $d/100$. Applying Proposition~\ref{prop:forest} to $X,Y$ with $\delta := d/100$ implies the existence of the desired paths $P_2, \ldots, P_r$.

To finish, note that this proposition gives us, together with $P_1$ and $\mathcal{C}$, a $\left(O \left(n/\delta^{1/5}\right),l,500 \lambda \right)$-good collection of paths. Since $t \geq n/2 \log n$, $l\ge t/7 \ge n/14 \log n$ and $d \geq (\log n)^{11}$, we have $O \left(n/\delta^{1/5}\right) \le \frac{l}{100\log n}$, as desired. 
\end{proof}
As stated before, Theorem \ref{thm:maindense} is an easy corollary of the above.
\begin{proof}[ of Theorem \ref{thm:maindense}]
Let $G$ be an $(n,d,\lambda)$-graph with $\lambda < \alpha d/10^{10}$, $d \geq n^{\alpha}$ and $n \geq 10000$. We claim that $G$ contains at least $\alpha n/8 $ vertex-disjoint cycles of size at most $4/ \alpha $. Indeed, we can do this in a greedy manner - suppose that we have vertex-disjoint cycles $C_1, \ldots, C_r$ of size at most $4/ \alpha$. Then, if $r < \alpha n/8 $, the set $X$ of vertices not contained in any of these cycles has size at least $n/2$. By part \textit{(2)} of Lemma~\ref{lem:expandermixing} it must then be that $G[X]$ has at least $\frac{d|X|^2}{2n} - \frac{\lambda n}{2} \geq \frac{nd}{10} > n^{1+\alpha/2}$ edges. By Theorem~\ref{BSthm}, $G[X]$ contains a cycle $C_{r+1}$ of length at most $4/ \alpha$ and this can be added to the collection of vertex-disjoint cycles.

We can then apply Theorem \ref{thm:pflhamiltonian} to $G$ since $\alpha n/8 \geq 20000 \lambda n/d$, thus showing that $G$ is Hamiltonian.
\end{proof}
\noindent We finish this section by noting that clearly Theorem~\ref{thm:pflhamiltonian} has the slightly more general corollary than Theorem~\ref{thm:maindense} concerning different densities of our pseudorandom graph. Its proof is the same as that of Theorem~\ref{thm:maindense}. 
\begin{cor}\label{thm:based}
Let $G$ be an $(n,d,\lambda)$-graph with $d \geq (\log n)^{11}$ and $d/\lambda > 10^{10}\log_{d} n$. Then, $G$ is Hamiltonian.
\end{cor}
\noindent Indeed, note that in general, it is easy to see that greedily we can always find a collection of $n/4 \log_d n$ vertex-disjoint cycles. Therefore, if  $d/\lambda > 10^{10}\log_{d} n$, then the conditions of Theorem \ref{thm:pflhamiltonian} are satisfied.  

\section{Applications}\label{sec:apps}

In this section, we discuss in more detail the applications of our new techniques and main results, that we mentioned briefly in Section \ref{sec:appsintro}.

\subsection{Random Cayley graphs and Hamilton cycles with few colours}\label{sec:random cayley}

Let $G$ be a group and $A\subseteq G$ a subset. The \emph{Cayley graph} $\Gamma(G,A)$ is the graph with vertex set $G$ and all edges of the form $\{g,ga\}$ for $g\in G,a\in A$. Note that we consider here \emph{undirected} Cayley graphs, hence, $\Gamma(G,A)$ is a $d$-regular graph with $|A|\le d=|A\cup A^{-1}| \le 2|A|$. We will assume here that $A$ does not contain the neutral element, so that $\Gamma(G,A)$ has no loops. As discussed in Section~\ref{sec:appsintro}, studying the Hamiltonicity of Cayley graphs (Conjecture~\ref{conj:transitive}) is a very important partial case of the Lov\'asz conjecture mentioned in Section~\ref{sec:intro}.
Up to now, only special cases of Conjecture~\ref{conj:transitive} have been proven, most notably when the ambient group is Abelian or when $A$ is linear in the size of $G$ (see Christofides, Hladk\'y and M\'ath\'e \cite{christofides2014hamilton}). But the general case is still wide open. In view of this, a lot of research has been invested into understanding random Cayley graphs, where $G$ is a group of order $n$ and $A\subseteq G$ of size $d=d(n)$ is chosen uniformly at random, and the question is whether $\Gamma(G,A)$ is Hamiltonian with high probability (as $n\to \infty$).

Alon and Roichman~\cite{AR:94} proved that for $d=O(\log n)$, whp $\Gamma(G,A)$ is connected (in fact, they proved the much stronger statement that it is an expander). This bound is tight up to a constant factor, for instance, in $G=\mathbb{Z}_2^n$, the smallest size of a set $A$ such that $\Gamma(G,A)$ is connected is $\log_2 |G|$. The analogous question for Hamiltonicity has turned out to be much more difficult. As a first result, Meng and Huang~\cite{MH:96} showed that almost all Cayley graphs are Hamiltonian. 
By using the technique of Alon and Roichman to bound the second eigenvalue of a random Cayley graph, Krivelevich and Sudakov~\cite{KS:03} deduced from their bound on the Hamiltonicity of $(n,d,\lambda)$-graphs that already $O(\log^5 n)$ random generators suffice for Hamiltonicity. Their technique for applying the Alon and Roichman result was later refined by Christofides and Markstr\"{o}m \cite{CM:12} to show that $O(\log^3 n)$ random generators suffice. It was also noted that Conjecture~\ref{conj:KS} would imply that $O(\log n)$ generators suffice, which seems reasonable since the random Cayley graph is then connected as discussed above, and there is no obvious obstacle to Hamiltonicity. This was also stated as a conjecture by Pak and Radoi\v{c}i\'{c}~\cite{PR:09}.

\begin{conj}\label{conj:cayley random}
Let $C$ be a sufficiently large absolute constant.
Let $G$ be a group of order $n$ and $d=C\log n$. If $A\subseteq G$ is a set of size $d$ chosen uniformly at random, then with high probability, $\Gamma(G,A)$ is Hamiltonian. 
\end{conj}

\noindent Using an operator Hoeffding inequality together with our new Theorem~\ref{thm:main general}, we can improve the number of random generators to $O(\log^{5/3}n)$.

\begin{thm}\label{thm:cayley random}
Let $C$ be a sufficiently large absolute constant.
Let $G$ be a group of order $n$ and $d=C\log^{5/3}n$. If $A\subseteq G$ is a set of size $d$ chosen uniformly at random, then with high probability, $\Gamma(G,A)$ is Hamiltonian. 
\end{thm}

\noindent A closely related problem concerns the quest for finding Hamilton cycles with few colours in optimally edge-coloured complete graphs. This is a proper edge-colouring of $K_n$
which uses $n-1$ colours when $n$ is even and $n$ colours when $n$ is odd. The following conjecture was proposed by Akbari, Etesami, Mahini, and Mahmoody~\cite{AEMM:07}.

\begin{conj}\label{conj:Ham colors}
Every properly edge-coloured $K_n$ with $\chi'(K_n)$ colours has a Hamilton cycle with $O(\log n)$ colours.
\end{conj}

\noindent Akbari et al. proved that one can find a Hamilton cycle with at most $8\sqrt{n}$ colours. The bound on the number of colours was later improved to $O(\log^3 n)$ by Balla, Pokrovskiy and Sudakov~\cite{BPS:17}. Their strategy was to pick $O(\log^3 n)$ colours at random and show that the subgraph induced by these colours is already Hamiltonian with high probability.
Here, we obtain a further improvement to $O(\log^{5/3} n)$. 

\begin{thm}\label{thm:Ham colors}
Every properly edge-coloured $K_n$ with $\chi'(K_n)$ colours has a Hamilton cycle with $O(\log^{5/3} n)$ colours.
\end{thm}

\noindent We prove both Theorems~\ref{thm:cayley random} and~\ref{thm:Ham colors} in a unified way. Note that when $n$ is even, then an optimal edge-colouring of $K_n$ is a partition of the edge set into perfect matchings. Moreover, if $G$ is a group and $a\in G$ is not the neutral element, then the edges in the Cayley graph corresponding to $a$ and $a^{-1}$ form either a $1$-factor (if $a=a^{-1}$) or a $2$-factor (if $a\neq a^{-1}$).
It turns out that whenever $K_n$ is partitioned into spanning regular graphs and we choose sufficiently many of them randomly, then with high probability, the obtained graph is pseudorandom and hence Hamiltonian.

\begin{thm}\label{thm:operator}
For all $R>0$ there exists $C>0$ such that the following holds:
Assume $K_n$ is edge-partitioned into regular spanning subgraphs $H_1,\dots,H_t$, with degrees $1\le r_i\le R$ for all $i\in[t]$. Let $c_1,\dots,c_k$ be a sequence of indices chosen independently from $[t]$ such that $i$ is always chosen with probability at least $1/Rt$.
If $k\ge C (\log n)^{5/3}$, then with high probability, $G=\cup_{j\in [k]}H_{c_j}$ is Hamiltonian.
\end{thm}

\begin{proof}
Since we do not require the $c_j$'s to be uniformly distributed, we first reduce to this case by introducing a dummy variable $*$ and a straightforward coupling.
For each $j\in [k]$ independently, we define $c_j'$ to be a random element from $[t]\cup\{*\}$ as follows. If $c_j=i$ then $c'_j=i$ with probability $\frac{1/\mathbb{P}[c_j=i]}{Rt}$ and 
$c'_j=*$ with probability $1-\frac{1/\mathbb{P}[c_j=i]}{Rt}$. Clearly this is well defined since $\mathbb{P}[c_j=i] \ge 1/Rt$ for all $i$. Thus we have that
$\mathbb{P}[c_j'=i]=1/Rt$ for all $i\in[t]$ and $\mathbb{P}[c_j'=*]=1-1/R$. 
Let $A^{(j)}$ be the adjacency matrix of $H_{c_j}$ with probability $\frac{n-1}{n}$ and $A^{(j)}=r_{c_j}I$ with probability $1/n$. Moreover, let $Y_j$ be the $0$-matrix if $c_j'=*$ and $Y_j=\frac{1}{4R}(A^{(j)}-\frac{r_{c'_j}}{n}J)$ otherwise, where $J$ is the all $1$'s matrix.
Crucially, we have $\mathbb{E}[Y_j]=0$. Indeed, the sum of the adjacency matrices of all $H_i$ is $J-I$ and $c'_j$ is uniformly distributed over $[t]$ if we condition on $c_j'\neq *$. On the other hand, if $c_j'=*$ then
$Y_j=0$. 
Moreover, note that $Y_j$ is a symmetric matrix with all eigenvalues in $[-1/2,1/2]$. Indeed, note that $A^{(j)}$ has eigenvalues in $[-R,R]$ since $H_{c_j}$ is either the adjacency matrix of a graph with maximum degree at most $R$ or $r_{c_j}I$; thus, $\frac{1}{4R}A^{(j)}$ has eigenvalues in $[-1/4,1/4]$. Also, clearly $\frac{1}{4R}\cdot \frac{r_{c'_j}}{n}J$ has eigenvalues in $[-1/4,1/4]$.

Using this setup for each individual $j\in [k]$, we now let $I\subseteq [k]$ be the set of $j$ for which $c_j'=c_j$.
Obviously, $G'=\cup_{j\in I}H_{c_j}$ is a subgraph of $G$, so it suffices to establish Hamiltonicity of $G'$.
We set further $A=\sum_{j\in I} A^{(j)}$ and $d=\sum_{j\in I}r_{c_j}$.
Observe that $X_i=Y_1+\dots+Y_i$ is a martingale and $X_k=\frac{1}{4R}(A-\frac{d}{n}J)$.
Using an operator Hoeffding inequality for Hilbert spaces, developed by Christofides and Markstr\"om~\cite{CM:08} (see \cite[Theorem~4]{BPS:17} for the version we use), one can show that for any $0<\eps<1/2$, it holds that
\begin{align}
\mathbb{P}[||X_k||\ge \eps k]\le 2n\exp(-2\eps^2 k).\label{operator chernoff}
\end{align}
We may assume that $k=\lceil C (\log n)^{5/3} \rceil$.
Plugging $\eps=\sqrt{\frac{\log n}{k}} \leq \frac{1}{\sqrt{C}(\log n)^{1/3}}$ into \eqref{operator chernoff}, we see that with high probability, $$||A-\frac{d}{n}J||=4R||X_k||\le 4\eps Rk.$$
Moreover, with high probability, the $c_j$'s with $j\in I$ are pairwise distinct and $A^{(j)}\neq r_{c_j}I$. Finally, by a simple Chernoff bound, we have $|I|\ge k/2R$ with high probability.

We assume now that all these events hold. Then, crucially, $A$ is the adjacency matrix of the random graph $G'$ and $\lambda(G')=||A-\frac{d}{n}J|| \le 4\eps Rk$.
So we can conclude that $G'$ is an $(n,d,\lambda)$-graph with $d = |I| \ge k/2R$ and $d/\lambda \ge 1/8 \eps R^2$. Since $\eps \leq \frac{1}{\sqrt{C}(\log n)^{1/3}}$, by choosing $C$ large enough, we can apply our Theorem~\ref{thm:main general} to establish Hamiltonicity of~$G'$.
\end{proof}

\noindent  We can now easily deduce the two results stated above.

\begin{proof}[ of Theorem \ref{thm:cayley random}]
Let $G$ be a group of order $n$. Instead of choosing $A$ of order $d$ uniformly at random, we may choose $d$ elements $a_1,\dots,a_d\in G$ independently and uniformly at random. For each $a\in G$ which is not the neutral element, let $H_a$ be the spanning subgraph of the complete graph on $G$ with all edges of the form $\{g,ga\}$ for $g\in G$. As discussed earlier, $H_a=H_{a^{-1}}$ is either a $1$-factor (if $a=a^{-1}$) or a $2$-factor (if $a\neq a^{-1}$). Each such factor is chosen with probability $\Theta(1/n)$ when we pick $a_i$, hence Theorem~\ref{thm:operator} implies the claim.
\end{proof}

\begin{proof}[ of Theorem \ref{thm:Ham colors}]
Balla, Pokrovskiy and Sudakov~\cite{BPS:17} found a reduction for odd $n$ to the case when $n$ is even. Thus, it suffices to consider the case when $n$ is even. Then, an optimal edge-colouring of $K_n$ is simply a partition into perfect matchings. We choose $O(\log^{5/3} n)$ of these at random. With high probability, the graph obtained is Hamiltonian, in particular, there is a Hamilton cycle with only $O(\log^{5/3} n)$ colours.
\end{proof}

\subsection{Additive patterns in multiplicative subgroups}

As we mentioned in Section \ref{sec:appsintro}, we can also apply our results to a problem of Alon and Bourgain~\cite{AB:14} on additive patterns in multiplicative subgroups. 
It is well-known that any multiplicative subgroup $A$ of the finite field $\mathbb{F}_q$ of size at least $q^{3/4}$ must contain two elements $x,y$ such that $x+y$ also belongs to $A$. Alon and Bourgain~\cite{AB:14} studied more complex additive structures in multiplicative subgroups and proved that when $A$ as above has size even $|A|\ge q^{3/4} (\log q)^{1/2-o(1)}$, then there is a cyclic ordering of the elements of $A$ such that the sum of any two consecutive elements is also in $A$. We improve on this result by showing that the additional polylog-factor can be avoided. 

\begin{thm}
There exists an absolute positive constant $c$ so that for any prime power $q$ and for any multiplicative subgroup $A$ of the finite field $\mathbb{F}_q$ of size $|A|\ge c q^{3/4}$ there is a cyclic ordering of the elements of $A$ such that the sum of any two consecutive elements is also in $A$.
\end{thm}

\noindent The proof is essentially the same as that of Alon and Bourgain~\cite[Theorem~1.2]{AB:14}, we only need to apply our new Theorem~\ref{thm:maindense} instead of the previous Krivelevich and Sudakov bound.

\begin{proof}
Let $c$ be a sufficiently large constant. 
Let $G$ be the graph with vertex set $\mathbb{F}_q$ where $xy\in E(G)$ if and only if $x+y\in A$ (this is the so-called \emph{Cayley sum graph}). By definition, $G$ has $q$ vertices, is $|A|$-regular and it was shown in \cite{AB:14} (see Lemma~2.7) that its second in absolute value eigenvalue is at most $q^{1/2}$. Now, let $H=G[A]$. It is easy to check that $H$ is regular as well. Let $d$ denote the degree of the vertices in $H$. Applying Lemma~\ref{lem:expandermixing} to $G$, we see that $e_G(A)\ge |A|^3/2q - q^{1/2}|A|$, so by choosing $c$ large  enough we have $d\ge |A|^2/2q$. Moreover, by interlacing of eigenvalues, it follows that $H$ is a $(|A|,d,q^{1/2})$-graph. Since $d\ge \frac{1}{2}c^{2}q^{1/2}$, we can then apply Theorem~\ref{thm:maindense} to obtain a Hamilton cycle in $H$, which yields the desired structure.
\end{proof}
\subsection{Classification-free proof of a result of Pak and Radoicic}\label{sec:pak}

As mentioned in Section \ref{sec:appsintro}, Pak and Radoi\v{c}i\'{c}~\cite{PR:09} proved that every finite group $G$ of size at least $3$ has a generating set $S$ of size $|S|\le \log_2 |G|$ such that the corresponding Cayley graph $\Gamma(G,S)$ contains a Hamilton cycle. 
However, their proof relies on the Classification of Finite Simple Groups, and motivated by the results discussed in Section~\ref{sec:random cayley}, they asked if there exists a classification-free proof of their result.
We essentially prove this, though our result is not quite as sharp since we can only guarantee a set $S$ of size $O(\log |G|)$. For the proof, we use a combination of a random and a deterministic choice of generators. As in Section~\ref{sec:random cayley}, we will use the fact that $O(\log |G|)$ random generators form with high probability a pseudorandom graph. The decisive advantage which we have here is to deterministically pick, say, three generators which will give us linearly many vertex-disjoint triangles. Together with an additional argument we can then facilitate a direct application of Theorem~\ref{thm:hamcycleplf}.

\begin{thm}
    Every group $G$ of order $n$ has a generating set $S$ of size $O(\log n)$ such that $\Gamma(G,S)$ is Hamiltonian.
\end{thm}

\begin{proof}
Let $C$ be a sufficiently large constant.
We let $S$ consist of three parts $S_1,S_2,S_3$. First,
let $a,b\in G$ be distinct with $b\neq a^{-1}$, and set $S_1=\{a,b,(ab)^{-1}\}$. We claim that $\Gamma(G,S_1)$ contains $n/10$ vertex-disjoint triangles. Indeed, suppose that $C_1,\dots,C_r$ are vertex-disjoint triangles and let $X\subseteq G$ be the set of vertices covered by those. If $r<n/10$, then there is $g\in G$ such that $g,ga,gab\notin X$, which gives us an additional triangle that we can add to the collection. Let from now on $t := n/10$ denote the number of triangles.

Next, let $S_2\subseteq G$ be a random set of size $d = C \log n$.  
Proceeding analogously to the proof of Theorem~\ref{thm:operator}, with high probability, $\Gamma(G,S_2)$ contains an $(n,d,\lambda)$-graph $H$ with $\lambda\le \eps d$ for $\eps := 10 ^{-10}$.
In particular, such a choice exists, which we will fix from now on.

We will now proceed similarly as in the proof of Theorem~\ref{thm:pflhamiltonian}. 
Analogously to Lemmas~\ref{lem:degreesoutsidecycles}--\ref{lem:reservoir}, we can find a path $P$ which connects a subcollection $\mathcal{C}$ of at least $t/7 \geq n/70$ many vertex-disjoint triangles found above with the following properties:
\begin{itemize}
\item $V_{P,\mathcal{C}}$ is clean in $H$, has size $t/7\geq n/70$ and thus spans a subgraph with minimum degree at least $d|V_{P,\mathcal{C}}|/2n \geq d/200$;
\item the endpoints of $P$ have degree at least $d/200$ into $V_{P,\mathcal{C}}$;
\item every vertex has at least $d/10$ neighbours in $Y$, where $Y$ is the complement of $V(P)\cup V(\mathcal{C})$.
\end{itemize}

It remains to cover $Y$ with a linear forest with few paths all of whose endpoints have many neighbours in $V_{P,\mathcal{C}}$.
We will first deal with the vertices that are not suitable endpoints.
Let $X$ be the set of vertices which have fewer than $d/200 \leq d|V_{P,\mathcal{C}}|/2n$ neighbours in $V_{P,\mathcal{C}}$. By Lemma~\ref{lem:cleaning}, $|X|\le 4\lambda^2 n^2/d^2|V_{P,\mathcal{C}}| \leq n/10^{17}$. Hence, we can apply Lemma~\ref{lem:coverforest} with $\delta := d/200$ to obtain a linear forest $\mathcal{F}_0$ which covers $X$ such that all its paths have their endpoints outside~$X$. We 
let $\mathcal{F}$ denote the union of this linear forest together with the path $P$ and regarding all remaining vertices as trivial paths of length~$0$.
This means we have almost found the desired good collection of paths. The only problem is that the number of paths in this linear forest is potentially too big (for example, it can be linear in $n$). To this end, we will now \emph{combine} the paths of $\mathcal{F}$ to reduce their number to at most $n/10000\log n$. By \emph{combining}, we simply mean that we can add an edge between the endpoints of two paths to replace them with a longer path. Crucically, this preserves the property that the endpoints of the paths are ``good'' in the sense that they still have large degree into $V_{P,\mathcal{C}}$.

Let $\mathcal{F}$ be the current linear forest. We claim that there exists a set $S_3$ of $O(\log n)$ generators such that using the edges of $\Gamma(G,S_3)$, we can combine paths of $\mathcal{F}$ to obtain a linear forest with at most $n/10000\log n$ paths. The idea is to find $S_3$ iteratively by always adding a generator which allows to combine as many paths as possible. We will show that this strategy performs as desired by analysing it in dyadic steps. To this end, we claim that if we have found a linear forest with at most $n/k$ paths, for some $k$, then we can choose at most $10k$ additional generators to reduce the number of paths to at most $n/2k$. Indeed, as long as the number of paths is at least $n/2k$, give each path a direction and let $A$ be the set of all starting vertices and $B$ the set of all terminal vertices (if a path consists only of one vertex, we add a copy of that vertex to both sets). Now, consider the bipartite graph with classes $A,B$ where two vertices are adjacent if and only if they are ends of distinct paths. Let $m=|A|=|B|\ge n/2k$. This graph has $m^2-m$ edges, each corresponding to an edge in the complete graph on~$G$. By averaging, there exists a generator $a$ such that $\Gamma(G,\{a\})$ contains at least $(m^2-m)/n$ edges from this auxiliary graph. We pick such a generator. Note that in this set of edges, we can pick at least a third of all edges to form a matching of size at least $n/10k^2$. Adding these matching edges to our collection of paths, we find a new linear forest and the number of paths has been reduced by at least $n/10k^2$. Repeating this procedure for at most $10k$ steps will yield a linear forest with at most $n/2k$ paths, as desired.

Now, we can simply apply the above claim iteratively with $k=1,2,4,\dots,2^j$ to see that a linear forest with at most $n/2^j$ paths has been found when the set $S_3$ has size at most 
$10\left(1+2+4+\dots+2^{j}\right)\le 20\cdot 2^j$.

Finally, we can apply Theorem~\ref{thm:hamcycleplf} to conclude that, setting $S:=S_1\cup S_2\cup S_3$, the Cayley graph $\Gamma(G,S)$ is Hamiltonian.
\end{proof}
\section{Concluding remarks}
In this paper we developed several new techniques and results concerning the Hamiltonicity of pseudorandom graphs. The main open problem here is to prove that every $(n,d,\lambda)$-graph with $d/\lambda\ge C$, for some universal constant $C>0$, has a Hamilton cycle. We proved this in the case that $d$ is a small polynomial of $n$. The key idea was to find linearly many vertex-disjoint cycles that we could arrange into an absorbing structure. While this proof method does not immediately extend to the sparser setting, we believe that many of the techniques we developed could play a crucial role in future work concerning Conjecture~\ref{conj:KS}. In particular, the use of the Friedman--Pippenger tree embedding technique with rollbacks and our method to cover a small set of `bad' vertices with paths having their endpoints outside the bad set seem very useful for a potential proof of Conjecture~\ref{conj:KS} in full. 

In the general case, we were able to reduce the required ratio $d/\lambda$ to order $\log^{1/3} n$. In particular, our results can be applied to pseudo-random graphs with sub-logarithmic degree. One very important example of such graphs
are random graphs. As we discussed in the introduction, the Hamiltonicity of random graphs is well understood. Particularly, the celebrated result of P\'osa \cite{posa:76} states that the random graph $G(n,p)$ is Hamiltonian with high probability when $p \geq C \log n/n$. In addition to the rotation-extension technique, his proof relies on so-called `booster' edges, which is where the randomness of the graph is crucially used. Our arguments in the proof of Theorem~\ref{thm:main general} can in fact be slightly adapted to prove the following statement, which gives
another proof of P\'osa's result. 
\begin{thm}\label{thm:lastthm}
Let $G$ be an $n$-vertex graph satisfying part \textit{(1)} of Lemma~\ref{lem:expandermixing} with $\lambda = \frac{d}{C(\log n)^{1/3}}$ for some large constant $C$ and such that every vertex has degree between $d/100$ and $100d$. Then, $G$ is Hamiltonian.
\end{thm}
\noindent This provides a proof of P\'osa's result which uses no \emph{direct} randomness. Indeed, it is not difficult to check that $G(n,p)$ for $p \geq 100 \log n/n$ satisfies the conditions of Theorem~\ref{thm:lastthm} with high probability.

Finally, we note that the case of presudo-random graphs with constant degree $d$ remains widely open. In fact, even when $G$ is  $(n,d,\lambda)$-graph which is almost optimally pseudorandom (i.e., $\lambda=O(\sqrt{d})$), it is not known whether it contains a nearly spanning path, i.e., a path of length $n - o(n)$.
It is not too difficult to show that $G$ contains a path of length $(1-\eps)n$ for some small constant $\eps=\eps_d$. Indeed, one can use P\'osa's rotation-extension technique to show that $(n,d,\lambda)$-graphs always contain paths of length at least $\left(1- \frac{100\lambda^2}{d^2}\right)n$ provided that $d > 10 \lambda$. However, it is not clear how this can be improved and solving this problem is a natural step towards proving Conjecture~\ref{conj:KS}.

\vspace{0.3cm}
\noindent
{\bf Acknowledgments.}
The authors would like to thank Nemanja Dragani\'c for fruitful discussions regarding the Friedman-Pippenger technique.

\end{document}